\documentclass{article}

\usepackage{psfrag}
\usepackage{graphicx}
\usepackage{amsmath,amsfonts,amssymb,amsxtra,amsthm,mathrsfs,enumerate}
\usepackage{pinlabel}
\usepackage{color}
\usepackage{tikz-cd}
\usepackage{hyperref}

\usepackage{wrapfig}

\usepackage[normalem]{ulem}

\usepackage{scalerel,stackengine}
\stackMath
\usepackage{verbatimbox} 
\usepackage{xparse}
\newlength\glyphwidth
\newlength\widthofx
\newsavebox\hatglyphCONTENT
\sbox\hatglyphCONTENT{%
    \addvbuffer[-0.05ex -1.3ex]{$\hat{\phantom{.}}$}%
}
\newcommand\hatglyph{\resizebox{0.6\widthofx}{!}{\usebox{\hatglyphCONTENT}}}
\newcommand\shifthat[2]{%
    \stackengine{0.2\widthofx}{%
        \SavedStyle#2}{%
        \rule{#1}{0ex}\hatglyph}{O}{c}{F}{T}{S}%
}
\ExplSyntaxOn
\newcommand\relativeGlyphOffset[1]{%
    \str_case:nnF{#1}{%
        {A}{0.18}%
        {B}{0.1}%
        {W}{0.02}%
        {J}{0.18}%
        {\phi}{0.17}%
    }{0.05}
}\ExplSyntaxOff
\NewDocumentCommand{\hatt}{mO{#1}}{%
    \ThisStyle{%
        \setlength\glyphwidth{\widthof{$\SavedStyle{}\longleftarrow$}}%
        \setlength\widthofx{\widthof{$\SavedStyle{}x$}}%
        \shifthat{\relativeGlyphOffset{#2}\glyphwidth}{#1}%
  }%
}

\DeclareMathOperator{\rank}{rk}
\DeclareMathOperator{\valence}{val}

\DeclareMathOperator{\im}{im}

\def\immerses{\looparrowright}
 \def\into{\hookrightarrow}

\def\ncl#1{\mathord{\langle}\mskip -4mu plus 0mu minus 0mu
  \mathord{\langle}#1\mathord{\rangle}\mskip -4mu plus 0mu minus 0mu
  \mathord{\rangle}}

\def\pb{P}
\def\hgraph{{\Gamma}}
\def\adj{W}
\def\und{{\Gamma_W}}
\def\folded{{\Gamma_W^I}}
\def\ambient{{\Omega}}

\def\real#1{\boldsymbol #1}
\def\rk#1{\rank(#1)}
\def\pio#1{\pi_1(#1)}
\def\subscript{x}

\def\GG{G} 
\def\HH{H}
\def\KK{K} 
\def\FF{F}

\def\NN{\mathbb{N}}
\def\ZZ{\mathbb{Z}}

\def\define{\raisebox{0.3pt}{\ensuremath{:}}\negthinspace\negthinspace=}

\newtheorem{theorem}{Theorem}[section]
\newtheorem*{thm}{Theorem}
\newtheorem{lemma}[theorem]{Lemma}
\newtheorem{corollary}[theorem]{Corollary}

\newtheorem{proposition}[theorem]{Proposition}
\newtheorem{conjecture}[theorem]{Conjecture}
\newtheorem{question}[theorem]{Question}

\theoremstyle{definition}
\newtheorem{definition}[theorem]{Definition}

\theoremstyle{remark}
\newtheorem{remark}[theorem]{Remark}

\newtheorem{example}[theorem]{Example}

\author{Larsen Louder and Henry Wilton}

\title{Negative immersions for one-relator groups}
\newcommand{\Addresses}{{
  \bigskip
  \footnotesize

  L.~Louder, \textsc{Department of Mathematics, University College London, Gower Street, London WC1E  6BT, UK}\par\nopagebreak
  \textit{E-mail address:} \texttt{l.louder@ucl.ac.uk}

  \medskip

  H. Wilton, \textsc{DPMMS, Centre for Mathematical Sciences, Wilberforce Road, Cambridge CB3 0WB, UK}\par\nopagebreak
  \textit{E-mail address:} \texttt{h.wilton@maths.cam.ac.uk}

}}

\begin{document}
\maketitle

\begin{abstract}
  We prove a freeness theorem for low-rank subgroups of one-relator
  groups. Let ${\FF}$ be a free group, and let $w\in{\FF}$ be a
  non-primitive element. The \emph{primitivity rank} of $w$, $\pi(w)$,
  is the smallest rank of a subgroup of ${\FF}$ containing $w$ as an
  imprimitive element. Then any subgroup of the one-relator group
  $\GG={\FF}/\ncl{w}$ generated by fewer than $\pi(w)$ elements is
  free. In particular, if $\pi(w)>2$ then $\GG$ doesn't contain any
  Baumslag--Solitar groups.

  The hypothesis that $\pi(w)>2$ implies that the presentation complex
  $X$ of the one-relator group $\GG$ has \emph{negative immersions}:
  if a compact, connected complex $Y$ immerses into $X$ and
  $\chi(Y)\geq 0$ then $Y$ Nielsen reduces to a graph.
  
  The freeness theorem is a consequence of a dependence theorem for
  free groups, which implies several classical facts about free and
  one-relator groups, including Magnus' Freiheitssatz and theorems of
  Lyndon, Baumslag, Stallings and Duncan--Howie.
  
  The dependence theorem strengthens Wise's $w$-cycles conjecture,
  proved independently by the authors and Helfer--Wise, which implies
  that the one-relator complex $X$ has \emph{non-positive immersions}
  when $\pi(w)>1$.
\end{abstract}


\section{Introduction}

\subsection{One-relator groups}
The beginnings of combinatorial group theory are often identified with
Dehn's articulation of the word, conjugacy and isomorphism problems
\cite{dehn_unendliche_1911}, and Magnus' solution of the word problem
for one-relator groups was an early triumph of the subject
\cite{magnus_identitatsproblem_1932}. The contemporary approach to
these decision problems takes the geometric route: to solve them in a
class of groups $\mathcal{C}$, one first shows that the groups in
$\mathcal{C}$ admit some kind of geometric structure. The fundamental
example is the class of word-hyperbolic groups, for which the word,
conjugacy and isomorphism problems have all been solved. Related
techniques can be applied to handle other important classes:
3-manifold groups, sufficiently small-cancellation groups and fully
residually free groups, to name a few.

After a century of progress, it is remarkable that the class of
one-relator groups is still almost untouched by geometric techniques,
and the conjugacy and isomorphism problems remain wide open. Many
one-relator groups are word-hyperbolic -- all one-relator groups with
torsion, and a randomly chosen one-relator group is $C'(1/6)$ -- but
there is also a menagerie of non-hyperbolic examples, including
Baumslag--Solitar groups, Baumslag's example
\cite{baumslag_non-cyclic_1969}, fundamental groups of two-bridge knot
complements, and the recent examples of Gardam--Woodhouse
\cite{gardam_geometry_2019}.

In this paper, we present theorems about the structure of one-relator
groups which begin to suggest a general geometric classification.  A
combinatorial map of 2-complexes $Y\to X$ is called an
\emph{immersion} if it is locally injective.  The starting point for
these results is a recent result established independently by the
authors \cite{louder-wilton} and by Helfer--Wise \cite{helfer-wise}:
the presentation complex ${X}$ of a torsion-free one-relator group has
\emph{non-positive immersions}, meaning that every connected, finite
2-complex ${Y}$ that immerses into ${X}$ either has $\chi({Y})\leq 0$
or Nielsen reduces\footnote{See Definition \ref{defn: Nielsen
    reduction} for the definition of Nielsen reduction.  For now it
  suffices to know that Nielsen reduction is stronger than homotopy
  equivalence .}  to a point.  We investigate the negatively curved
analogue of this definition.

\begin{definition}\label{defn: Negative immersions}
  A compact 2-complex $X$ has \emph{negative immersions} if, for every
  immersion from a compact, connected 2-complex $Y$ to $X$, either
  $\chi(Y)<0$ or $Y$ Nielsen reduces to a graph.
\end{definition}

On the face of it, negative immersions should be a difficult condition
to check, since it applies to all immersed compact complexes $Y$.
However, there turns out to be a connection with a quantity defined by
Puder \cite{puder_primitive_2014}

\begin{definition}\label{defn: Primitivity rank}
  Let ${\FF}$ be a free group and $w\in{\FF}$.  The \emph{primitivity
    rank} of $w$ is
  \[ 
  \pi(w)=\min\{\rk{\KK} \mid w\in \KK\leq{\FF}\mbox{ and }w\mbox{ not
    primitive in }\KK\}\in\NN\cup\{\infty\}~,
  \]
  where, by convention, $\pi(w)=\infty$ if $w$ is primitive in
  ${\FF}$, since in that case $w$ is primitive in every subgroup $\KK$
  containing $w$.  Note that $\pi(1)=0$, since $1$ is an imprimitive
  element of the trivial subgroup.
\end{definition}

The first main theorem of this paper tells us that negative immersions
for the presentation complex ${X}$ of a one-relator group
$\GG={\FF}/\ncl{w}$ are governed by $\pi(w)$.

\begin{theorem}[Negative immersions for one-relator groups]
  \label{thm: Negative immersions}
  The presentation complex of the one-relator group ${\FF}/\ncl{w}$
  has negative immersions if and only if $\pi(w)>2$.
\end{theorem}

Thus, negative immersions can be determined in practice for the
presentation complexes of one-relator groups. There is an algorithm to
compute the primitivity rank $\pi(w)$ -- see Lemma \ref{lem: Finitely
  many w-subgroups} -- and furthermore it is often easy to compute it
by hand for small examples by considering how a map representing $w$
factors through immersions.
 
\begin{example}\label{eg: Concrete example}
  Let $w=acaba^{-1}b^{-1}c^{-1}\in \langle a,b,c\rangle=F_3$.  As
  usual, if $F_3$ is realised as the fundamental group of a bouquet
  $\Omega$ of 3 circles then $w$ can be represented by an immersion
  $S\immerses \Omega$ where $S$ is a graph homeomorphic to a circle.
  Represent a subgroup $\KK\leq\FF$ containing $w$ by a finite graph
  $\Gamma$ such that the map $w$ factors as
 \[
 S\immerses\Gamma\immerses\Omega\,.
 \]
 If the edges labelled by $a$ in $S$ are not all identified in
 $\Gamma$ then one such edge is only crossed once, which implies that
 $w$ is primitive in $\KK$. The same holds for $b$ and $c$, and this
 in turn implies that the map $\Gamma\to\Omega$ is the
 identity. Therefore $\pi(w)\geq 3$.  Furthermore, the Whitehead graph
 of $w$ is connected without cut vertices, so $w$ is not itself
 primitive. Therefore $\pi(w)=3$.
 \end{example}

The primitivity ranks of all words of length at most 16 in the free
group of rank 4 were computed by Cashen--Hoffmann
\cite{cashen_short_2020}.

Puder \cite[Corollary 8.3]{puder_expansion_2015} proved that a generic
word $w$ in a free group of rank $k$ has $\pi(w)=k$, so in particular,
when $\rank\FF>2$, a generic one-relator complex has negative
immersions.  Theorem \ref{thm: Negative immersions} follows from Lemma
\ref{lem: Negative immersions}, which is a finer classification of
immersions from complexes with sufficiently large Euler
characteristic.

Non-positive immersions constrains the subgroup structure of a
group. Recall that a group $G$ is called \emph{coherent} if every
finitely generated subgroup is finitely presented. Non-positive
immersions implies a homological version of coherence: if $X$ has
non-positive immersions, then the second homology group of any
finitely generated subgroup of $\pi_1X$ is finitely generated
\cite[Corollary 1.6]{louder-wilton}. Indeed, Wise conjectured that the
fundamental groups of complexes with non-positive immersions are
coherent. The authors and, independently, Wise, have shown that
one-relator groups with torsion are coherent
\cite{louder_one-relator_2018,wise_coherence_2018}.

Our next theorem asserts that negative immersions also constrain the
subgroup structure of a one-relator group. Recall that a group $G$ is
called \emph{$k$-free} if every subgroup generated by $k$ elements is
free. The \emph{rank} of a group $G$ is the minimal number of elements
needed to generate $G$, and is denoted $\rk{G}$.  {Note that, if $G$
  has a one-relator presentation with $n$ generators, then either $G$
  is free (of rank $n-1$) or $\rk G=n$; see Remark \ref{rem:
    One-relator rank} below.}

\begin{theorem}[Low-rank subgroups of one-relator groups]
  \label{thm: f.g. subgroups}
  \label{thm: k-free}
  Let $\GG={\FF}/\ncl{w}$ be a one-relator group with
  $\pi(w)>1$. There is a finite collection $P_1,\dotsc,P_n$ of freely
  indecomposable, one-relator subgroups of $\GG$, each of rank
  $\pi(w)$, with the following property.  Let $H\leq\GG$ be a finitely
  generated subgroup.
  \begin{enumerate}[(i)]
  \item If $\rk{H}<\pi(w)$ then $H$ is free.
  \item If $\rk{H}=\pi(w)$ then $H$ is either free or conjugate
    into some $P_i$.
  \end{enumerate}
  In particular, the one-relator group $\GG$ is $(\pi(w)-1)$--free.
\end{theorem}

The $P_i$ are defined in Subsection~\ref{subsection: primitivity rank
  and w subgroups}.  Theorem~\ref{thm: k-free} is a cousin of Magnus'
Freiheitssatz, which says that if $\HH$ is a proper free factor of a
free group ${\FF}$ and the natural map $\HH\to{\FF}/\ncl{w}$ is not
injective then $w$ is in fact conjugate into
$\HH$~\cite{magnus-freiheitssatz}.  Theorem~\ref{thm: k-free} follows
immediately from Lemma \ref{lem: Universal property of w-subgroups},
which applies to homomorphisms from groups of low rank to $\GG$.

As far as the authors are aware, Theorem \ref{thm: k-free} implies all
known $k$-freeness theorems for one-relator groups.  For instance,
combining Theorem \ref{thm: k-free} with \cite[Corollary
  8.3]{puder_expansion_2015} recovers the following theorem of
Arzhantseva--Olshanskii \cite{arzh-olsh}.

\begin{corollary}   \label{cor: k-free random}
A generic $k$--generator one-relator group is  $(k-1)$-free. 
\end{corollary}

Generic one-relator groups satisfy the $C'(1/6)$ small-cancellation
property. However, we emphasise that there are many words $w$ with
$\pi(w)>2$ that are not small-cancellation, such as Example \ref{eg:
  Concrete example}.

\begin{remark}\label{rem: pi is a group invariant}
  It follows immediately from Theorem \ref{thm: k-free} that $\pi(w)$
  is the minimal rank of a non-free subgroup of the one-relator group
  $G=\FF/\ncl{w}$. In particular, $\pi(w)$ is an isomorphism invariant
  of $G$.
\end{remark}

Taken together, Theorems \ref{thm: Negative immersions} and \ref{thm:
  k-free} imply that one-relator groups with negative immersions have
a similar subgroup structure to hyperbolic groups.

\begin{corollary}\label{cor: Subgroups of NI groups}
  Let $w$ be an element of a free group ${\FF}$. If the one-relator
  group $\GG={\FF}/\ncl{w}$ has negative immersions then $\GG$ doesn't
  contain any Baumslag--Solitar groups and any finitely generated
  abelian subgroup of $\GG$ is {cyclic}.
\end{corollary}

A famous question in geometric group theory asks whether or not a
group with a finite classifying space and without Baumslag--Solitar
subgroups must be hyperbolic \cite[Question
  1.1]{bestvina_questions_????}.  Lyndon's identity theorem implies
that presentation complexes of torsion-free one-relator groups are
classifying spaces, so in light of Corollary \ref{cor: Subgroups of NI
  groups}, the case of one-relator groups with negative immersions is
of immediate interest.

\begin{conjecture}\label{conj: One-relator hyperbolization}
  Every one-relator group with negative immersions is hyperbolic.
\end{conjecture}

A positive resolution of Conjecture \ref{conj: One-relator
  hyperbolization} would resolve the conjugacy and isomorphism
problems for the class of one-relator groups with negative immersions.
Of course, one can also ask whether one-relator groups with negative
immersions have other conjectural properties of hyperbolic groups,
such as residual finiteness and surface subgroups.
 
Since $\pi(w)=1$ if and only if the corresponding one-relator group
has torsion, and these are known to be hyperbolic by the B. B. Newman
Spelling Theorem~\cite{newman-spelling,hruska-wise-spelling}, the
remaining case of interest is $\pi(w)=2$.  One-relator groups with
primitivity rank two seem to behave differently than the rest; we
state a mild strengthening of Theorem~\ref{thm: f.g. subgroups} in
this case.

\begin{corollary}\label{cor: NN associate properties}
  Let $w$ be an element of a free group ${\FF}$. If $\pi(w)=2$ then
  the one-relator group $\GG={\FF}/\ncl{w}$ contains a subgroup
  $P\leq\GG$ with the following properties:
 \begin{enumerate}[(i)]
 \item $P$ is a two-generator, one-relator group;
 \item every two-generator subgroup of $\GG$ is either free or
   conjugate into $P$.
 \end{enumerate}
\end{corollary}

We call the subgroup $P$ the \emph{peripheral subgroup} of $\GG$ (we
cannot currently prove that $P$ is an isomorphism invariant of $\GG$).
We are unable to say anything new about two-generator one-relator
groups -- note that, in this case, Corollary~\ref{cor: Subgroups of NI
  groups} is vacuous and Corollary \ref{cor: NN associate properties}
is trivially true.

\begin{example}
  Take $\KK$ to be a rank-two free factor $\langle x,y\rangle$ of a
  free group $\FF$, and let $w=[x,y]$. In this case, $G=\FF/\ncl{w}$
  is the free product of $P=\KK/\ncl{w}\cong\ZZ^2$ together with a
  complementary free-group factor, and Corollary \ref{cor: NN
    associate properties} implies that every freely indecomposable
  two-generator subgroup of $G$ is conjugate into $P$. In this case,
  the conclusion also follows from the Kurosh subgroup theorem;
  however, in more complicated examples, $G$ will not split as a free
  product and the Kurosh subgroup theorem will not apply.
\end{example}

Corollary \ref{cor: NN associate properties} suggests the following
natural counterpart to Conjecture \ref{conj: One-relator
  hyperbolization}.

\begin{conjecture}\label{conj: rel hyp}
  Suppose $\pi(w)=2$. Then $\GG$ is hyperbolic relative to $P$.
\end{conjecture}

Conjectures~\ref{conj: One-relator hyperbolization} and~\ref{conj: rel
  hyp} provide a conceptual explanation for the fact that all known
examples of pathological one-relator groups have two generators.

\subsection{The dependence theorem}

In 1959, Lyndon proved that a non-trivial commutator in a free group
$\FF$ cannot be expressed as a square \cite{lyndon-abc}.  In this
paper, we view Lyndon's theorem as the first in a line of
\emph{dependence theorems} for free groups, which bound the rank of
the target of a homomorphism in which certain elements are forced
either to be conjugate or to have roots.

\begin{thm}[Lyndon, 1959]
  Let $\HH=\langle a,b\rangle$, $v=[a,b]$, and consider the group
  $G=\HH*_{v=x^n}\langle x\rangle$ where $n=2$. If $f\colon G \to F$
  is a surjective homomorphism onto a free group then $\rk{F}\leq 1$.
\end{thm}

\begin{remark}\label{rem: Dyck's theorem}
  Lyndon in fact proved this theorem for the word $v'=a^2b^2$. This is
  equivalent to the theorem for $v=[a,b]$, by Dyck's theorem.
\end{remark}

Shortly afterwards, the hypotheses of Lyndon's theorem were weakened
to cover the case when $n\geq 2$; see, for
example,~\cite[Lemma~36.4]{baumslag-aspects}.  The commutator
$v=[a,b]$ in Lyndon's theorem cannot be replaced by an arbitrary
element of the free group; indeed, adjoining a root to a generator $a$
exhibits a map in which the rank of the target group does not go down.
We therefore need a hypothesis that excludes generators.  Recall that
a collection of subgroups $\{M_1,\dotsc,M_n\}$ of a group $H$ is
called \emph{malnormal} if $M_i\cap hM_jh^{-1}\neq 1$ implies that
$i=j$ and $h\in M_i$, for any indices $i,j$ and $h\in H$.

\begin{definition}
  \label{def: independent dependent}
  A malnormal collection of cyclic subgroups $\{\langle v_j\rangle\}$
  of a group $H$ is called \emph{independent} if there exists a free
  splitting $\HH=\HH'*\langle v_k\rangle$ of $\HH$, for some $k$, with
  $v_j$ conjugate into $\HH'$ for $j\neq k$.  Otherwise, $\{\langle
  v_j\rangle\}$ is called \emph{dependent}.
\end{definition}

Note that a singleton $\{\langle v\rangle\}$ in a free group $H$ is
dependent if and only if $v$ is not primitive. Using the theory of
pro-$p$ groups, Baumslag generalized Lyndon's theorem to all dependent
malnormal singletons $\{\langle v\rangle\}$ \cite{baumslag}.

\begin{thm}[Baumslag, {1965}]
  Let $\HH$ be a free group, $\{\langle v\rangle\}$ dependent
  (i.e.\ not a primitive element) and malnormal (i.e.\ not a proper
  power) in $H$ and $n>1$.  If $G=\HH*_{v=x^n}\langle x\rangle$ and
  $f\colon G \to F$ is a surjective homomorphism onto a free group,
  then $\rk{F}<\rk{\HH}$.
\end{thm}

We now introduce the data for a more general dependence theorem.  Let
$\HH_1,\ldots, \HH_l$ be free groups and $\{\langle
v_{i,j}\rangle\}_{i=1\cdots l,j=1\cdots m_i}$ a malnormal collection
of non-trivial cyclic subgroups of $\HH_i$.  For each $i$ and $j$, let
$n_{i,j}$ be a positive integer. We associate a graph of groups
$\Delta=\Delta(\{H_i\},\langle x \rangle,\{\langle
v_{i,j}\rangle\},\{n_{i,j}\})$ to these data as follows. There are $l$
vertices labelled by the $\HH_i$, arranged around one central vertex
labelled $\langle x\rangle$. For each $i$ and $j$, there is an edge
which attaches the subgroup $\langle v_{i,j}\rangle$ to the
index-$n_{i,j}$ subgroup of the vertex group $\langle x\rangle$ via
the homomorphism mapping $v_{i,j}$ to $x^{n_{i,j}}$.\footnote{When
  $l=1$ or $m_i=1$ we will drop the indices $i$ or $j$ as appropriate,
  to minimize notation.}

A dependence theorem relates these data to the rank of a possible free
image of $\pio{\Delta}$.  For instance, Lyndon's theorem is the case
when $l=m=1$, $\HH=\langle a,b\rangle$, $v=[a,b]$ and $n=2$.  A more
general theorem of this form can be proved using the techniques of
\cite{adjoiningroots} (cf.\ Theorems 1.3 and 1.5 of that paper).
 
\begin{thm}[Louder, 2013]
  Let $\HH_1,\ldots, \HH_l$ be free groups, $\{\langle
  v_{i,j}\rangle\}$ a malnormal collection of non-trivial cyclic
  subgroups of $\HH_i$ and $n_{i,j}$ positive integers.  Let $\Delta$
  be the associated graph of groups and let $f\colon
  \pio{\Delta}\to\FF$ be a surjective homomorphism to a free group
  with $f\vert_{\HH_i}$ injective for each $i$.  If the family
  $\{\langle v_{i,j}\rangle\}$ is dependent for each $i$, and
  $\sum_{i,j} n_{i,j}>1$, then
  \[
    \rk{\FF}-1<\sum_{i}(\rk{\HH_i}-1)~.
  \]
\end{thm}

Baumslag's theorem, and hence Lyndon's, follows immediately. Indeed,
if $f\vert_{\HH}$ is not injective, the conclusion holds
automatically, and otherwise the theorem applies.  A 1983 theorem of
Stallings in a similar spirit also follows {\cite[Theorem
    5.3]{stallings-surfaces}}; we discuss Stallings' theorem in
Subsection \ref{subsec: stallings}.

Another kind of dependence theorem constrains the integers $n_{i,j}$
in terms of the ranks of the $\HH_i$. A prototypical result here is
provided by a theorem of Duncan and Howie, which extends and
quantifies Lyndon's theorem by bounding from below the genus of a
proper power \cite{duncan-howie}.  {A special case of the}
Duncan--Howie theorem can be stated as follows.

\begin{thm}[Duncan--Howie, 1991]
  Let $\Sigma$ be a compact, orientable surface of genus $g$ with one
  boundary component; let $H=\pio{\Sigma}$ and let $\langle
  v\rangle=\pio{\partial\Sigma}$.  Let $\Delta$ be the graph of groups
  obtained by adjoining an $n$th root to $v$
  (i.e.\ $\Delta=\Delta(H,\langle x\rangle,\langle v\rangle,n)$) and
  $f\colon\pio{\Delta}\to F$ be a homomorphism onto a free group with
  $f(v)\neq 1$. Then $n \leq\rk{H}-1=2g-1$.
\end{thm}

Just as Lyndon's theorem was generalized from surfaces to more general
dependent malnormal families of cyclic subgroups, so the Duncan--Howie
theorem can be extended to arbitrary dependent malnormal families of
cyclic subgroups.  The following theorem, proved by the authors and
also Helfer--Wise, answered Wise's \emph{$w$-cycles conjecture}, which
was made in connection with the question of whether or not one-relator
groups are coherent \cite{helfer-wise,louder-wilton}.

\begin{thm}[Louder--Wilton, Helfer--Wise]
  Let $\HH$ be a free group, $\{\langle v_j\rangle\}$ a malnormal
  collection of non-trivial cyclic subgroups of $\HH$ and $n_j$
  positive integers.  Let $\Delta$ be the associated graph of groups
  and let $f\colon \pio{\Delta}\to\FF$ be a homomorphism to a free
  group with $f\vert_{\HH}$ injective.  If the family $\{\langle
  v_j\rangle\}$ is dependent then $\sum_{j} n_j\leq \rk{\HH}-1$.
\end{thm}

Despite the fifty-nine years of work documented above, there are
simple examples that do not fall within the scope of these
theorems. For instance, consider the next example (which famously
demonstrates that stable commutator length does not coincide with
commutator length in free groups).

\begin{example}\label{eg: Non-sharpness}
  Let $\Sigma$ be a torus with one boundary component. Let
  $F=\pi_1\Sigma=\langle a,b\rangle$ and let $w=[a,b]$ correspond to
  the boundary component.  Consider the homomorphism $F\to S_3$ given
  by $a\mapsto (2 3)$ and $b\mapsto (1 2)$, so $w\mapsto (1 2 3)$, and
  let $\Sigma'\to \Sigma$ be the 3-sheeted covering map corresponding
  to the stabiliser of $1$ in $S_3$.  The boundary component unwraps
  three times in $\Sigma'$ and therefore, by computing Euler
  characteristic, $\Sigma'$ is a surface of genus two with a single
  boundary component, represented by $v\in \HH=\pi_1(\Sigma')$.

  In summary, in this example, $\rk{F}=2$, $\rk{H}=4$, $\langle
  u\rangle$ is dependent and malnormal in $\HH$, and the inclusion
  $\HH\to \FF$ sends $v\mapsto w^3$.
\end{example}

Example \ref{eg: Non-sharpness} is, of course, consistent with the
theorems of Baumslag, the first author and Duncan--Howie. However, the
first two theorems only assume that $v\mapsto w^n$ with $n>1$, and
conclude that $\rk{F}\leq 3$. Likewise, the Duncan--Howie theorem
asserts that $n\leq 3$, but places no constraint on $\rk{F}$.
Intuitively, one is lead to conjecture a common generalisation, which
imposes an upper bound on $n+\rk{F}$.

\begin{theorem}\label{introthm: main}
  Let $\HH_1,\ldots,\HH_l$ be free groups, $\{\langle
  v_{i,j}\rangle\}$ a malnormal collection of non-trivial cyclic
  subgroups of $\HH_i$ and $n_{i,j}$ positive integers.  Let $\Delta$
  be the associated graph of groups and let $f\colon
  \pio{\Delta}\to\FF$ be a surjective homomorphism to a free group
  with $f\vert_{\HH_i}$ injective for each $i$.  Then
  \[
\rk{\FF}-2+    \sum_{i,j} n_{i,j}\leq
    \sum_i(\rk{\HH_i}-1)
  \]
  if the family $\{\langle v_{i,j}\rangle\}$ is dependent for each $i$.
\end{theorem}

We do not know if the inequality of Theorem~\ref{introthm: main} is
sharp; see Question~\ref{sharp question}.

As stated, Theorem \ref{introthm: main} does not strictly generalize
the Duncan--Howie theorem, since the map $f$ in Theorem \ref{introthm:
  main} is required to be injective on the $\HH_i$. Theorem
\ref{maintheorem} relaxes the injectivity hypothesis to a hypothesis
of `diagrammatic irreducibility', which is weak enough to encompass
the Duncan--Howie theorem; see Corollary \ref{dhcorollary} for
details.

The connection between the dependence theorem and one-relator groups
goes via an estimate on the Euler characteristic of the one-relator
pushout of a branched map; the reader is referred to Definitions
\ref{def: branched map} and \ref{def: one-relator pushout} for
the relevant terms.  A special case of the estimate can be stated as
follows, which is a direct consequence of Corollary
\ref{monotonicity}.

\begin{corollary}
  \label{cor: Intro monotonicity}
  Let $f\colon Y\immerses X$ be an immersion from a compact, connected
  two-complex $Y$ to the presentation complex $X$ of a one-relator
  group $G=F/\ncl{w}$, with $w$ not a proper power. If $Y$ has no free
  faces then
  \[
    \chi(Y)\leq\chi(\hatt Y)~,
  \]
  where $\hatt{Y}$ is the one-relator pushout of $f$.
\end{corollary}

As well as the applications to non-positive immersions mentioned
above, this estimate on Euler characteristics also gives new proofs of
Magnus' Freiheitssatz and Lyndon's asphericity theorem; see Theorem
\ref{thm: Magnus Lyndon}.

\subsection{Remarks about Theorem \ref{introthm: main} and its proof}

The proof of Theorem \ref{introthm: main} combines \emph{stackings}
(first defined in \cite{louder-wilton}) with \emph{adjunction spaces}
(the main tool of \cite{adjoiningroots}).  The definition of a
stacking was inspired by the proof of the Duncan--Howie theorem, which
in turn draws on the theory of one-relator groups developed by Magnus
and others, as well as the tower argument of Papakyriakopoulos.

The adjunction space is the natural topological representative of the
graph of groups $\Delta$ considered in the previous section. The map
$\pi_1(\Delta)\to F$ can be realised by a \emph{second}
graph-of-spaces structure on the adjunction space, and the rank of the
underlying graph of gives an upper bound for $\rk{F}$ -- see Remark
\ref{rem: pi1 surjectivity}.  In a nutshell, the idea of the proof is
now to use the stacking constructed in \cite{louder-wilton} to analyse
the homology of the adjunction space in a manner reminiscent of Morse
theory.  The stacking enables us to define fibrewise filtrations on
the adjunction space.  An analysis of these filtrations reduces the
proof of the main theorem to a combinatorial lemma -- the up-down
lemma of \S\ref{Ss: Updown lemma}.

The resulting inequality, that of Theorem \ref{introthm: main},
combines the degrees of the adjunctions with the rank of the
target. This improves on both the Duncan--Howie theorem (and its
generalisation in \cite{louder-wilton}), which only sees the degrees
of the adjunctions, and the theorems of Baumslag, Stallings and the
first author, which only see the rank of the target.  Morally, Theorem
\ref{introthm: main} can be thought of as a kind of non-abelian
rank-nullity theorem.

\subsection*{Acknowledgements}

The second author was funded by EPSRC Standard Grant EP/L026481/1. The
authors are grateful to the anonymous referees for their close reading
of the manuscript; their comments greatly improved the exposition.
The authors are also grateful to Jim Howie and Hamish Short for
pointing out an error in the the original version of Lemma \ref{lem:
  nielsenreduces}.

\section{Graphs and graphs of graphs}

\subsection{Graphs}
\label{graphs}

We start by recalling the basic set-up of graphs.

An (oriented) graph $G$ is a tuple $G=(V_G,E_G,\iota,\tau)$, where
$V_G$ and $E_G$ are sets, (called the \emph{vertices} and \emph{edges}
of $G$, respectively) and $\iota\colon E_G\to V_G$ and $\tau\colon
E_G\to V_G$ are maps (called \emph{incidence} maps). When convenient
we suppress the subscript $G$. We often use the letter $\alpha$ to
denote an incidence map, which might be either $\iota$ or $\tau$.

A \emph{morphism of graphs} is a pair of maps $f:V_G\to V_{G'}$ and
$f:E_G\to E_{G'}$, such that the natural diagrams
\begin{center}
  \begin{tikzcd}
 E_G\arrow{r}{\alpha}\arrow{d}{f} & V_G\arrow{d}{f} \\
 E_{G'} \arrow{r}{\alpha} & V_{G'} 
  \end{tikzcd}
\end{center}
commute, for $\alpha=\iota,\tau$.  A graph is \emph{simple} if its
edges are determined by their endpoints, i.e., if $\iota(e)=\iota(e')$
and $\tau(e)=\tau(e')$ then $e=e'$. Note that in this case, the map
$(\iota,\tau)$ naturally identifies $E_G$ with a subset of $V_G\times
V_G$.

If there is a partition $V_G = I_G\sqcup T_G$ such that
$\iota(E_G)\subseteq I_G$ and $\tau(E_G)\subseteq T_G$ then $G$ is
called \emph{bipartite}.  A \emph{morphism of bipartite graphs} is a
morphism of graphs that respects the bipartite structure.  Again, if
$G$ is a simple bipartite graph, then $(\iota,\tau)$ identifies $E_G$
with a subset of $I_G\times T_G$.

Given a graph $G$ the \emph{geometric realization} of $G$ is the
1-complex
\[
  \real{G}=(V_G\sqcup
  (E_G\times\left[-1,1\right]))/\{(e,-1)\sim\iota(e),(e,1)\sim\tau(e)\}~.
\]
We implicitly identify $V_G$ with its image in $\real{G}$.

The realization of a morphism of graphs $f\colon G\to G'$ is the map
\begin{equation*}
  \real{f}(x)=
  \begin{cases}
    f(x) & \mbox{ if } x\in V_G \\
    (f(e),t) & \mbox{ if } x=(e,t)\in e\times\left[-1,1\right]
  \end{cases}
\end{equation*}

The \emph{(Euler) characteristic} of a graph $G$ is defined to be
quantity
\[
\chi(G):=\vert V_G\vert-\vert E_G\vert \, .
\]
For a choice of $v_0\in V_G$, we define
$\pio{G,v_0}\define\pio{\real{G},v_0}$. As usual, we will suppress
mention of the base point $v_0$ when it will not cause confusion.  In
the usual way, a morphism of graphs $f\colon G\to G'$ induces a
homomorphism $f_*=\real{f}_*$ on fundamental groups.

The valence of a vertex $v\in V_G$ is defined to be
\[
\valence(v)\define\#\{e\in E_G\mid \iota(e)=v\}+ \#\{e\in E_G\mid \tau(e)=v\}~.
\]
If $\real{G}$ is homeomorphic to $S^1$ then we say that $G$ is a
\emph{cycle}; equivalently, $G$ is finite and connected, and every
vertex has valence two.

A morphism of graphs is an \emph{immersion} if for all $e\neq e'$ and
$\alpha\in\{\iota,\tau\}$, $\alpha(e)=\alpha(e')$ implies $f(e)\neq
f(e')$.  Note that the realization $\real{f}$ of an immersion $f$ is
locally injective, and by Stallings \cite[5.3]{stallings-folding}
induces an injective map $f_*$ on fundamental groups.

As in \cite{stallings-folding}, a finite graph is called a \emph{core
  graph} if there are no vertices of valence $1$.

\subsection{Combinatorial complexes}
\label{sec: combinatorial complexes}

We are now ready to define the class of 2-complexes that we will work
with.

\begin{definition}\label{def: Combinatorial complex}
  A \emph{combinatorial (2-dimensional) complex} $X$ is a tuple
  \[
  (G_X,S_X,w_X,o_X)
  \]
  where $G_X$ is a graph, $S_X$ is a disjoint union of cycles,
  $w_X:S_X\to G_X$ is an immersion of graphs and $o_X$ is an
  orientation on $\real{S}_X$. We emphasise that $o_X$ is not required
  to relate to the structure of $S_X$ as an oriented graph, and in
  general, it will not.
\end{definition}

As usual, we suppress subscripts when it will not cause confusion. We
will also often suppress mention of the orientation $o_X$ as well.

A \emph{morphism} of combinatorial complexes $f:X\to X'$ consists of a
map of graphs $f:G_X\to G_{X'}$ and an immersion $s:S_X\to S_{X'}$
such that the diagram
\begin{center}
  \begin{tikzcd}
 S_X\arrow{r}{s}\arrow{d}{w_X} & S_{X'}\arrow{d}{w_{X'}} \\
 G_X \arrow{r}{f} & G_{X'} 
  \end{tikzcd}
\end{center}
commutes and $\real{s}^*(o_{X'})=o_X$.  We emphasise that, to avoid
the notation becoming too burdensome, we will usually use the same
letter to denote the map of 2-complexes and the the map of 1-skeleta,
and introduce the letter $s$ to denote the accompanying map of
circles.

The \emph{realization} of a combinatorial complex $X$ is the space
\[
\real{X}\define \real{G}_X \sqcup (\real{S}_X\times [0,1])/\sim
\]
where $(\theta,0)\sim \real{w}(\theta)$ for all $\theta\in \real{S}_X$
and $(\theta_1,1)\sim(\theta_2,1)$ whenever $\theta_1$ and $\theta_2$
are in the same component of $\real{S}_X$.  Note that this definition
is functorial: a morphism $f$ of combinatorial complexes $X\to X'$
naturally induces a continuous map $\real{f}:\real{X}\to\real{X}'$.
As usual, given a choice of vertex $x_0$ in $G_X$, we define
$\pi_1(X,x_0)\define\pi_1(\real{X},x_0)$.

A \emph{free face} of a combinatorial complex $X$ is an edge $e\in
E_{G_X}$ such that $|w_{X}^{-1}(e)|=1$.  Note that if $X$ has a free
face, then $\real{X}$ can be simplified by a simple homotopy, which
collapses a 2-cell (see Section \ref{subsection: nielsen
  equivalence}).  The \emph{boundary} of $X$, $\partial X$, is the
union of its free faces. The \emph{boundary} of the realization,
$\partial\real{X}$, is the union of the closed edges corresponding to
$\partial X$.

The \emph{link} of a vertex $v$ of $G_X$ is a graph $L\equiv L(v)$
defined as follows.  The set of vertices of $L$ is
\[
V_L\define\{(x,\alpha)\in E_{G_X}\times\{\iota,\tau\}\mid \alpha(x)=v\}
\]
and the set of edges is $E_L\define w^{-1}(v)\subseteq V_{S_X}$.
Since $S_X$ is a disjoint union of cycles, for each $y\in E_L$ there
exist exactly two edges of $E_{S_X}$ incident in $S_X$ at $y$. Let
$i_y$ be the edge immediately preceding $y$ according to $o_X$ and let
$\alpha_y\in\{\iota,\tau\}$ such that $\alpha_y(i_y)=y$. Likewise, let
$t_y$ be the edge immediately following $y$ according to $o_X$ and let
$\beta_y(t_y)=y$.  Then $(w(i_y),\alpha_y)$ and $(w(t_y),\beta_y)$ are
both vertices of $V_L$, and we define the incidence maps of $L$ by
setting $\iota(y)=(w(i_y),\alpha_y)$ and $\tau(y)=(w(t_y),\beta_y)$.

A morphism of combinatorial complexes $f:X\to X'$ naturally induces
maps on links $f_v:L(v)\to L(f(v))$.

We can now define the class of morphisms that we are concerned
with. Informally, branched maps are maps that are locally
injective away from vertices and midpoints of 2-cells, and immersions
are locally injective everywhere.

\begin{definition}\label{def: branched map}
  A morphism $f:X\to X'$ of combinatorial complexes is a
  \emph{branched map} if every induced map on links $f_v$ is an
  immersion.  Furthermore, if every $f_v$ is injective and the map
  $s:S_X\to S_{X'}$ is injective on each component, we say that $f$ is
  an \emph{immersion}.  In this case, we write $f:X\immerses X'$.
\end{definition}

In combinatorial group theory, a special role is played by \emph{van
  Kampen diagrams} -- planar 2-complexes that represent relations in
the fundamental group.  These can be seen as special cases of
morphisms of combinatorial complexes.

\begin{definition}\label{def: vK diagram}
Let $X$ be a combinatorial complex.  A \emph{van Kampen diagram} over
$X$ is pair of morphisms
\[
S\to D\to X
\]
where $S$ is a cycle and $\real{D}\cup_{\real{S}} D^2$ is homeomorphic
to the 2-sphere $S^2$.  If $S\to D$ is an immersion and $D\to X$ is a
branched map then the van Kampen diagram is said to be
\emph{reduced}.
\end{definition}

\begin{remark}\label{rem: vK diagrams slightly generalized}
  This notion of van Kampen diagram is slightly more general than the
  standard one (cf. \cite[Chapter III, \S 9]{lyndon-schupp}), since we
  allow branching over the centres of the 2-cells in $D$.
\end{remark}

The next lemma provides a useful characterization of branched
maps.

\begin{lemma}\label{lem: branched maps and pullbacks}
A morphism $f$ of combinatorial complexes
\begin{center}
  \begin{tikzcd}
 S_X\arrow{r}{s}\arrow{d}{w_X} & S_{X'}\arrow{d}{w_{X'}} \\
 G_X \arrow{r}{f} & G_{X'} 
  \end{tikzcd}
\end{center}
is a branched map if and only if the map
\[
(w_X,s):E_{S_X}\to E_{G_X}\times E_{S_{X'}}
\]
is an embedding.
\end{lemma}
\begin{proof}
  The link of an edge $e$ in a combinatorial complex $X$ is the set
  $w_X^{-1}(e)$. For any vertex $v$ of $X$, the map $f_v:L(v)\to
  L(f(v))$ is an immersion if and only if it does not fold any pair of
  edges of $L(v)$. Equivalently, $f_v$ is an immersion if and only if
  $s$ restricts to an injective map on the star of every vertex of
  $L(v)$. Therefore, $f$ is a branched map if and only if it
  induces injective maps on the links of edges of $X$.  To complete
  the proof of the lemma, note that $s$ is injective on each link
  $w_X^{-1}(e)$ if and only if the map $(w_X,s)$ is injective.
\end{proof}

\subsection{Graphs of graphs and the adjunction space}

The construction below appears in various guises in the
papers~\cite{dicks-shnc,louder-mcreynolds,adjoiningroots}.

\begin{definition}
  \label{def: graph of graphs 2}
  A graph of graphs is a graph $M=(V_M,E_M,\varphi,\psi)$, where the
  vertex and edge sets $V_M$ and $E_M$ are collections of graphs. The
  elements of $V_M$ are the \emph{vertex graphs} of $M$, and the
  elements of $E_M$ are the \emph{edge graphs} of $M$. The incidence
  maps $\varphi$ and $\psi$ are considered as maps from $V_M$ to $E_M$
  as sets, but which specify morphisms of graphs. That is, given $E\in
  E_M$ there are morphisms of graphs $\varphi_E:E\to \varphi(E)\in
  V_M$ and $\psi_E:E\to\psi(E)\in V_M$.
\end{definition}

Given a graph of graphs $M$, we construct a new graph $\Gamma_M$,
which is the graph $M$ with the additional data forgotten. There is a
natural ``forgetful'' (iso)morphism of graphs $\pi:M\to \Gamma_M$. For
each $x\in E_{\Gamma_M}\sqcup V_{\Gamma_M}$, we let $M_x$ denote the
graph $\pi^{-1}(x)$. In what follows we keep this
notation. Furthermore, if $M_e\in E_M$ is an edge graph then
$\varphi(M_e)\subseteq M_v\in V_M$, for some $v=\varphi(e)\in
V_{\Gamma_M}$, and to avoid unnecessary subscripts, we denote the map
$\varphi_{M_e}$ simply by $\varphi$. Likewise for $\psi$. In this way
we think of $M$ as being a family of graphs ``indexed'' by the graph
$\Gamma_M$, which is just the graph of graphs $M$ with its additional
structures stripped away.

A morphism of graphs of graphs $f\colon M\to M'$ is a morphism of
graphs, where the map $f$ is endowed with additional structures
compatible with the graph of graphs structure, i.e., for each $x\in
\Gamma_M$, there is a morphism of graphs $f_x\colon M_x\to M_{f(x)}$
which is compatible with edge maps.  That is, for each $e\in
E_{\Gamma_M}$ the following squares commute.
\begin{center}
  \begin{tikzcd}
    M_e\arrow{r}{\varphi_e}\arrow{d}{f_e} & M_{\varphi(e)}\arrow{d}{f_{\varphi(e)}} & &     M_e\arrow{r}{\psi_e}\arrow{d}{f_e} & M_{\psi(e)}\arrow{d}{f_{\psi(e)}}\\
    M'_{f(e)}\arrow{r}{\varphi_{f(e)}} & M'_{f(\varphi(e))} & &     M'_{f(e)}\arrow{r}{\psi_{f(e)}} & M'_{f(\psi(e))}
  \end{tikzcd}
\end{center}

The set of vertex graphs $V_M$ is the collection of graphs
{$\{M_v\}_{v\in V_{\Gamma_M}}$}, and we abuse notation and denote the
disjoint union $\coprod M_v$ by $M_V$. Likewise for edge graphs. The
realization $\real{M}$, defined below, has an additional graph of
graphs structure $M'=(\{M_V\},\{M_E\},\varphi,\psi)$ with one vertex
space and one edge space, i.e., the underlying graph $\Gamma_{M'}$ has
only one vertex and one edge, with edge maps $\varphi$ and
$\psi$. Clearly $\real{M}'$ is naturally homeomorphic to
$\real{M}$. This additional graph of graphs structure will turn up
again in Subsection~\ref{subsection: computing the characteristic},
where it is used in a Mayer--Vietoris argument.

\begin{definition}\label{def: Realisation of a graph of graphs}
The \emph{realization} of a graph of graphs $M$ is the space
\[
\real{M}\define \real{M}_V\sqcup (\real{M}_E\times [-1,1])/\sim
\]
where $(x,-1)\sim \varphi(x)$ and $(x,+1)\sim\psi(x)$ for all $x\in\real{M}_E$.
\end{definition}

The notion of realization allows us to define fundamental groups and
other topological invariants of $M$ as usual. The Euler characteristic
will play a particularly important role.

\begin{definition}\label{def: Euler characteristics}
  The \emph{(Euler) characteristic} of a graph of graphs $M$ is 
  \[
  \chi(M) \define \chi(M_V)-\chi(M_E)~.
  \]
  Note that $\chi(M)$ is the Euler characteristic of $\real{M}$.
\end{definition}

\begin{remark}
  Any oriented graph $G=(V_G,E_G,\iota,\tau)$ can be regarded as a
  graph of graphs by regarding $V_G$ and $E_G$ as graphs with no
  edges, and regarding the edge maps $\iota$ and $\tau$ as morphisms
  of graphs. In this case the indexing graph has one edge and one
  vertex, with edge maps $\iota$ and $\tau$. In this case the
  forgetful map returns the graph itself. Likewise, a bipartite graph
  may be regarded as a graph of graphs indexed by a graph with two
  vertices and one edge.
\end{remark}

We are particularly interested in graphs of graphs with connected
vertex and edge spaces.  Given a graph of graphs $M$, there is a
canonical $\hatt M$ with connected vertex and edge spaces, and a
morphism $f:\hatt M\to M$ with the following property: For each $y\in
\Gamma_{M}$, the restriction
\[
f\colon\coprod_{x\in f^{-1}(y)}\hatt{M}_x\to M_y
\]
is an isomorphism of graphs.  The vertex graphs of $\hatt M$ are the
connected components of the vertex graphs of $M$, the edge graphs of
$\hatt M$ are the connected components of the edge graphs of $M$, and
the edge maps of $\hatt M$ are induced by restriction.  Furthermore,
the realisation $\real{f}:\hatt{\real{M}}\to\real{M}$ is a
homeomorphism.

Therefore, without loss of generality, we may always assume that
graphs of graphs have connected vertex and edge spaces, and we will do
so.

\begin{remark}\label{rem: pi1 surjectivity}
  If the vertex and edge graphs of $M$ are connected then the natural
  map $\pi_1(M)\to\pi_1(\Gamma_M)$ is surjective.
\end{remark}

In the terminology of Wise and his coauthors, the realization
$\real{M}$ is a $VH$-complex \cite{bridson_scr_1999}, and the maps
$\varphi$ and $\psi$ are the incidence maps of the \emph{vertical}
graph-of-graphs structure on $\real{M}$.  We now turn our attention to
the equally natural \emph{horizontal} graph-of-graphs structure on
$\real{M}$, called $W$.

Let $\iota$ and $\tau$ be the incidence maps coming from the defining
data of the graphs in $V_M$ and $E_M$. We define a new intermediate
graph of graphs structure, $N$, as follows: $N$ has one vertex graph
$V$ with vertex set $\sqcup_{v\in V_{\Gamma_M}}V_{M_v}$ and edge set
$\sqcup_{e\in E_{\Gamma_M}}V_{M_e}$, with edge maps induced by
$\varphi$ and $\psi$, and $N$ has one edge graph $E$ with vertex set
$\sqcup_{v\in V_{\Gamma_M}}E_{M_v}$ and edge set $\sqcup_{e\in
  E_{\Gamma_M}}E_{M_e}$, with edge maps again induced by $\varphi$ and
$\psi$. The collection of incidence maps $\tau$ and $\iota$ join
together to induce maps $\tau: E\to V$ and $\iota: E\to V$. We then
set $W=\hatt N$. The next lemma records the fact that this is an
alternative graph of graphs structure on $\real{M}$.

\begin{lemma}\label{lem: Vertical vs horizontal}
  The realizations $\real{W}$ and $\real{M}$ are homeomorphic.
\end{lemma}

\begin{remark}
  It is the fact that our graphs and their morphisms are oriented that
  automatically endows $\real{M}$ with the dual graph-of-graphs
  decomposition. This need not hold for graphs of unoriented graphs.
\end{remark}

\begin{remark}\label{rem: Gamma_W is a pushout}
  The underlying graph $\Gamma_W$ is  the pushout of the diagram 
  \begin{center}
    \begin{tikzcd}
      \coprod M_e\arrow[bend left = 10]{r}{\varphi}\arrow[swap, bend right = 10]{r}{\psi} & \coprod M_v  
    \end{tikzcd}
  \end{center}
  in the category of (oriented) graphs.
\end{remark}

A morphism of combinatorial complexes naturally defines a graph of
graphs -- the \emph{adjunction space}.

\begin{definition}\label{def: Adjunction space}
  Let $f:X\to X'$ be a morphism of combinatorial complexes.  The
  \emph{adjunction space} $M\equiv M(f)$ is the (bipartite) graph of
  graphs with vertex set
  \[
  V_M\define \{G_X, S_{X'}\}~,
  \]
  edge set
  \[
  E_M\define\{S_X\}
  \]
  and incidence maps given by $s:S_X\to S_X'$ and $w:S_X\to G_X$.
\end{definition}

\subsection{Resolving}\label{ss: Resolving}

We now specialize the discussion above to the case of interest for our
main theorem. We start with two combinatorial complexes and a map $h$
between them, defined by the following data.
\begin{center}
  \begin{tikzcd}
 P\arrow{r}{\sigma}\arrow{d}{\lambda} & S\arrow{d}{w} \\
 \Gamma \arrow{r}{h} & \Omega 
  \end{tikzcd}
\end{center}

Let $M$ be the adjunction space of $h$ and let $W$ be the horizontal
graph of graphs associated with $M$.  We now observe that this set-up
entails the existence of various natural maps, summarized in the
following commutative diagram.

\begin{center}
  \begin{tikzcd}
     & S\arrow[swap]{dr} \arrow{dr} \arrow[bend left=5]{drr}{w} \arrow[bend left=15]{drrr}{w}\\
     {\pb}\arrow[swap]{ur}{\sigma} \arrow{dr}{\lambda} & & {\adj} \arrow{r}{m} & {\und} \arrow{r}\arrow{r}{l} & {\ambient}\\
   & \hgraph \arrow{ur}\arrow[bend right=5]{urr} \arrow[bend right=15, swap]{urrr}{h}
  \end{tikzcd}
\end{center}

The morphism $h$ determines a map of graphs\footnote{But not of graphs
  of graphs.} ${\adj}\to {\ambient}$ that sends vertical vertex-graphs
to vertices and vertical edge-graphs to midpoints of edges.

We now \emph{resolve} this map, by factoring it through the underlying
graph $\Gamma_W$ of ${{\adj}}$.  The map ${\adj}\to {\ambient}$
factors through the natural map $m\colon {\adj}\to {\Gamma_W}$. There
are natural morphisms $S\to W$ and $\Gamma\to W$ which, when composed
with $m$, descend to morphisms $S\to\Gamma_W$ and
$\Gamma\to\Gamma_W$. By Remark \ref{rem: Gamma_W is a pushout}, the
maps $S\to\Omega$ and $\Gamma\to\Omega$ factor through a canonical map
of graphs $l:\Gamma_W\to\Omega$; the map of graphs $W\to \Omega$ then
factors as:
\[
  {\adj}\stackrel{m}{\longrightarrow} {\Gamma_W}\stackrel{l}{\longrightarrow} {\ambient}~.
\]
For the most part in what follows, this enables us to replace
${\ambient}$ by $\Gamma_W$, and for that reason, we will also denote
by $w$ the natural map $S\to {\und}$, even though, strictly speaking
$w$ is a map from $S$ to ${\ambient}$.

We denote by ${\folded}$ the graph obtained by Stallings folding the
map $l\colon {\und}\to {\ambient}$ to an immersion. Note that
$\chi({\folded})\geq\chi({\und})$.

\begin{remark}\label{rem: f factors through m}
  We can use this set-up to draw group-theoretic conclusions about
  $\pi_1(\Omega)$, since the induced homomorphism
  $\pio{\real{\adj}}\to\pio{\real{\ambient}}$ factors through the map
  $\real{m}_*\colon\pio{\real{\adj}}\to\pio{\real{\Gamma}_W}$.  Note
  that point preimages in $\real{m}$ are connected, and therefore
  $\real{m}_*$ is surjective.
\end{remark}

\begin{figure}
  \centerline{
    \includegraphics[width=.8\textwidth]{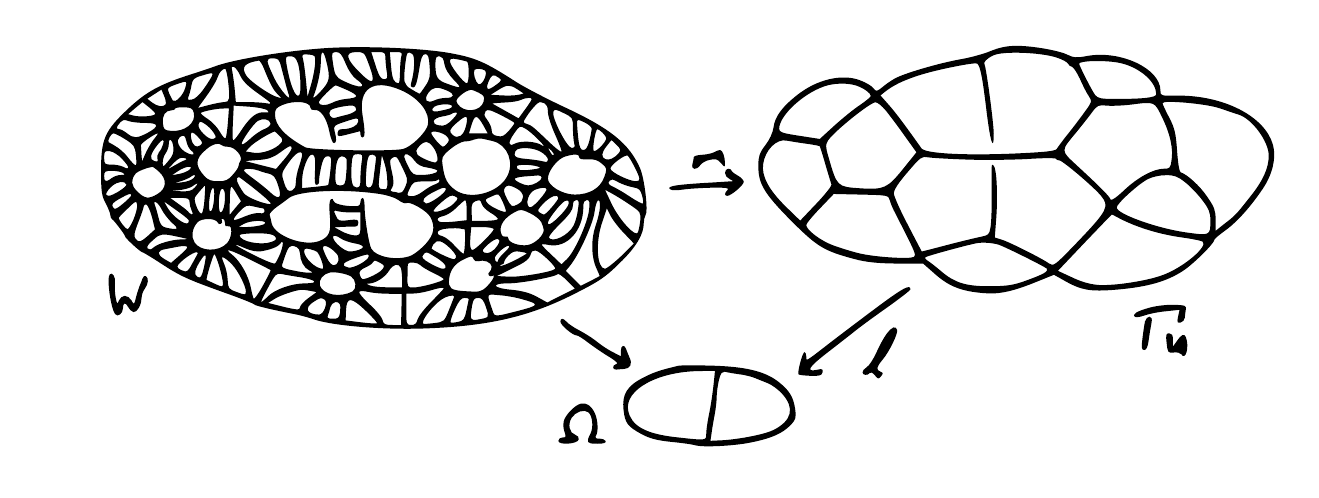}
  }
  \caption{Schematic of ${\adj}$. The realization $\real{\adj}$ is the
    graph of spaces obtained by gluing the ends of $\real{\pb}\times
    I$ to $\real{\hgraph}$ and $\real{S}$ using $\lambda$ and
    $\sigma$. We think of $\real{\hgraph}$ and $\real{S}$ as running
    through $\real{\adj}$ horizontally. The vertical graph-of-graphs
    structure on $\real{\adj}$ is cartoonishly depicted above, with
    vertex spaces $\real{\adj}_v$ and edge spaces $\real{\adj}_e$. The
    morphism of graphs $\adj\to\ambient$ factors as $l\circ m$, where
    $m$ is the projection from $\adj$ to its underlying graph $\und$.}
  \label{schematic}
\end{figure}

\subsection{The dependence theorem}

\label{subsec: dependence}

The boundary of $\adj$ consists of those edges of $\hgraph$ that are
hit by precisely one element of ${\pb}$.
  
\begin{definition}\label{defn: Boundary}
  Let ${\adj}$ be the horizontal graphs of graphs for the adjunction
  space defined above. The \emph{boundary} of ${\adj}$ is
  \[
    \partial{\adj}=
    \{e\in E_{\hgraph}\mid \vert\lambda^{-1}(e)\vert = 1\}~.
  \]
  The boundary of $\real{\adj}$ is
  \[
    \partial\real{\adj}=
      \bigcup_{{e\in\partial\adj}}
      \overline{e\times(-1,1)}
    \subseteq
    \real{\hgraph}~.
  \]
\end{definition}

By construction $\partial\adj$ is the boundary of the complex defined
by $\lambda:P\to\Gamma$. As mentioned above, when ${\adj}$ has
nonempty boundary, this complex can be simplified by a collapse.  We
call this circumstance \emph{independent} (since it implies the
group-theoretic notion of independence given in the introduction).  We
will also be interested in a strengthening of this, in which the whole
image of $S$ in ${\und}$ (and therefore in $\ambient$) is covered at
least twice by the boundary.

\begin{definition}\label{defn: Dependent}
  The map $\lambda\colon\pb\to\hgraph$ is \emph{independent} if
  $\partial\adj\neq \varnothing$; otherwise, it is called
  \emph{dependent}. The map $\lambda\colon\pb\to\hgraph$ is
  \emph{strongly independent {(over $\Omega$)}} if, for all $e\in
  w(E_S)$, $\vert\partial\adj\cap\adj_e\vert\geq 2$; otherwise, it is
  called \emph{weakly dependent {(over $\Omega$)}}.
\end{definition}

\begin{example}
  \label{ex: branched map weakly dependent}
  Let $\Omega$ be the (oriented) graph with one vertex $v$ and two
  edges $a$ and $b$, and let $X$ be the combinatorial complex
  $(\Omega, S, w, o)$, where $S$ is a cycle with three edges,
  $w:S\to\Omega$ is the immersion determined by the word $abb$, and
  $o$ is arbitrary. The realization of $X$ is the M\"obius strip, with
  one boundary edge labeled $a$. Let {$Y=(\Gamma,P,\lambda,o')$} be
  the combinatorial complex corresponding to the annular connected
  double cover of $X$ with the two lifts of the vertex identified
  {(with $o'$ pulled back from $o$)}. Then the boundary of $W$
  consists of two edges which both map to the edge $a$ in $\Omega$, so
  $\partial W$ does not (doubly) surject $w(E_S)=E_{\Omega}$, and in
  this case $\lambda$, which represents the attaching map for the two
  two-cells in $Y$, is weakly dependent. It is not, however dependent,
  since the boundary is non-empty.
\end{example}

We are interested in the setting where $h$ is a branched
map. This has the following consequences for $W$.

\begin{lemma}
  \label{lem: DI properties}
Let ${\adj}$ be the horizontal graphs of graphs for the adjunction
space associated to a branched map $h$.
  \begin{enumerate}[(i)]
  \item The graph ${\adj}_E$ is a simple bipartite graph.
  \item The incidence maps of $W$ are injective on the edges of each
    edge space $W_e$, and also on the set of vertices of $W_e$ that
    come from $S$.
  \end{enumerate}
 \end{lemma}
 
The proof is left as an easy exercise.

\begin{remark}\label{rem: Cycle valence}
  If $S$ is connected and $\sigma\colon {\pb}\to S$ is a covering map
  then, for each $s\in S_{\subscript}\subset {\adj}_{\subscript}$,
  $\valence(s)=\deg(\sigma)$.
\end{remark}

In our case, the complex defined by the map $w:S\to \Omega$ will be a
one-relator complex, meaning that $S$ is just a single cycle. We call
$w$ \emph{indivisible} if it does not factor through a proper covering
map $S\to S'$.

We can now state the dependence theorem in the form in which we prove
it.

\begin{theorem}[Dependence theorem]
  \label{maintheorem}
  Let $h$ be a branched map of combinatorial complexes as above,
  and let $W$ be the horizontal graph of spaces for the adjunction
  space.  Suppose further that $S$ is a single cycle and that $w\colon
  S\to {\ambient}$ is indivisible. If $\lambda\colon
  {\pb}\to \hgraph$ is weakly dependent then
  \[
    \chi(\hgraph)+\deg(\sigma)-1\leq\chi({\und})~.
  \]
\end{theorem}

Usually, following~\cite{stallings-folding}, subgroups of free groups
are represented by immersions of connected graphs, so for the purposes
of generalizing the theorems of Baumslag and Stallings it is safe to
restrict to immersions of connected $\hgraph\to{\ambient}$. However,
in order to strengthen the Duncan--Howie theorem we need to allow maps
that are not immersions.

\begin{example}
  \label{ex: need indivisible}
  Theorem~\ref{maintheorem} does not hold when the condition that $w$
  be indivisible is relaxed, as the following example illustrates. Let
  $\ambient$ be the graph with one vertex $v$ and one one-cell $e$,
  let $h\colon\Gamma\to\ambient$ and $w\colon S\to \ambient$ be the
  connected degree-two and degree-three covers, respectively, and let
  $P\to \ambient$ be the connected degree-six cover. Let
  $\lambda\colon P\to \Gamma$ be the degree-three cover and
  $\sigma\colon P\to S$ a degree-two cover. In this example,
  $\und\cong\ambient$, and $\lambda$ is weakly dependent, since each
  vertex and edge graph $W_e$ and $W_v$ is isomorphic to the complete
  bipartite graph $K_{2,3}$, but the theorem predicts
  \[
  0+2-1\leq 0
  \]
   which is, of course, false.
 \end{example}

If $h$ is a branched map and $\lambda$ is weakly dependent then
$\chi({\und})\leq -1$, and in this case Theorem \ref{maintheorem}
implies the inequality
\[
  \chi(\hgraph)+\deg(\sigma)\leq 0~,
\]
which is precisely Wise's $w$-cycles conjecture
\cite{helfer-wise,louder-wilton}.

Since we do not know if the inequality of Theorem \ref{maintheorem} is
sharp, we pose the question here.

\begin{question}
  \label{sharp question}
  Are there ${\adj}$ as above, with $\lambda$ dependent, such
  that
  \[
    \chi(\hgraph)+\deg(\sigma)-1=\chi({\und})
  \]
  for all $\deg(\sigma)\geq 2$ and $\chi({\und})\leq -1$? What about for
  $\lambda$ weakly dependent?
\end{question}

We next explain how Theorem \ref{maintheorem} implies
Theorem~\ref{introthm: main}.

\begin{proof}[Proof of Theorem \ref{introthm: main}]
  We may assume that $f(w)$ is not a proper power in $\FF$: if there
  are $v\in\FF$, $k\geq 1$, such that $f(w)=v^k$ then, since $f$ is
  surjective, if
  \[
  \rk{\FF}-2 +\sum_{i,j} kn_{i,j}\leq
    \sum_i(\rk{\HH_i}-1)\,,
  \]
  then certainly
  \[
  \rk{\FF}-2 +\sum_{i,j} n_{i,j}\leq
    \sum_i(\rk{\HH_i}-1)\,.
  \]

  We take ${\ambient}$ to be a rose with ${\FF}=\pio{\ambient}$ a free
  group, $\hgraph$ to be a graph immersing into ${\ambient}$, for
  which the components have fundamental groups $H_i$, and
  $\lambda\colon {\pb}\to \hgraph$ an immersion of a disjoint union of
  cycles into $\hgraph$ that represent the family $\{v_{i,j}\}$. Since
  each $f(v_{i,j})$ is conjugate into $\langle w\rangle$, these factor
  through a common cycle $w\colon S\to {\ambient}$ which induces the
  maps $\sigma\colon {\pb}\to S$ and $\lambda\colon {\pb}\to \hgraph$.
  We may therefore construct the adjunction space $M$ and its
  associated horizontal graph of graphs ${\adj}$.
  
  By definition, $\pio{\Delta}=\pi_1(M)$, which is in turn canonically
  isomorphic to $\pio{W}$ by Lemma \ref{lem: Vertical vs horizontal}.
  The map $\pio{W}\to \FF$ is surjective by Remark \ref{rem: pi1
    surjectivity} and factors through the surjection $m_*\colon
  \pio{\adj}\to \pio{\und}$ so
  $\chi(\FF)=\chi({\ambient})\geq\chi({\und})$.
  
  Because $\{\langle v_{i,j}\rangle\}$ is dependent, it follows that
  the map $\lambda\colon {\pb}\to \hgraph$ is dependent, in particular
  weakly dependent. Since $\{\langle v_{i,j}\rangle \}$ is malnormal,
  the natural map $P\to \Gamma\times_\Omega S$ is an embedding, so
  $(\Gamma,P)\to(\Omega,S)$ is a branched map of 2-complexes, by
  Lemma \ref{lem: branched maps and pullbacks}.  The result now
  follows from Theorem \ref{maintheorem}, after noting that
  \[
    \deg(\sigma) = \sum_{i,j} n_{i,j}\,,
  \] 
  that $\chi({\ambient})=1-\rk{\FF}$, and that
  $\chi(\hgraph)=\sum_i(1-\rk{\HH_i})$.
\end{proof}

\section{One-relator pushouts}

\label{subsec: one-relator pushouts}

We consider a branched map $f$ from a combinatorial complex $Y$
to a \emph{one-relator} complex $X$ -- that is, $X$ is a combinatorial
complex with a single 2-cell.  Suppose $w\colon {S}\to {\ambient}$ is
the attaching map defining $X$, and the map $\lambda\colon {\pb}\to
{\hgraph}$ defines $Y$. In this section we will see that the
realization of the pushout $\und$ of $\hgraph$ and $S$ along $\pb$ is
the one-skeleton of a ``best'' one-relator complex $\hatt Y$ that the
map $Y\to X$ factors through. The dependence theorem implies that when
$Y$ cannot be simplified in an obvious way, i.e. when $Y$ doesn't have
any free faces, then $\chi(Y)\leq\chi(\hatt Y)$.

The components $P_i$ of $P$ are the boundaries of the 2-cells of $Y$.
The \emph{degree of branching} of $P_i$ under $f$ is denoted by $n_i$,
and is the degree of the covering map $\sigma|_{P_i}:P_i\to S$.
Clearly
\begin{equation}
  \sum_i(n_i-1)=\deg(\sigma)-\#\{e\mid e\mbox{ is a two-cell in }Y\}~.\label{cells}
\end{equation}

\begin{definition}[One-relator pushout]
  \label{def: one-relator pushout}
   The one-relator complex defined by the map $w:S\to\Gamma_W$ is
   denoted by $\hatt Y$, and is called the \emph{one-relator pushout}
   of $Y$.  The map $l$ extends to a branched map $\hatt{Y}$ to
   $X$ (also denoted by $l$).  By Remark \ref{rem: Gamma_W is a
     pushout}, $\hatt{Y}$ has the following universal property:
   \begin{center}
     \begin{tikzcd}
       Y\arrow{rr}\arrow{dr}& & Z\arrow{r} & X\\
       & \hatt Y\arrow[dotted]{ur}{\exists!}\arrow[bend right=15]{urr}{l} & &
     \end{tikzcd}
   \end{center}
   whenever $Z\to X$ is a morphism of degree one.
 
   Similarly, the \emph{immersed one-relator pushout} $\hatt{Y}^I$ is
   the complex defined by the natural map $S\to\Gamma^I_W$.  It enjoys
   a similar universal property for immersions $Z\to X$ of degree one.
 \end{definition}

In the context of one-relator complexes, the dependence theorem gives
a relation between the Euler characteristics of $Y$ and $\hatt Y$.

\begin{corollary}[One-relator pushout inequality]
  \label{monotonicity}
  Let $f\colon Y\to X$ be a branched map from a compact
  combinatorial complex $Y$ to a one-relator complex $X$ defined by
  $w:S\to\Omega$, with $w$ indivisible. If the restriction
  $f\vert_{\partial Y}\colon\partial Y\to w(E_S)$ is not at least
  two-to-one then
  \[
    \chi(Y)+\sum_i(n_i-1)\leq\chi(\hatt Y)\leq \chi(\hatt Y^I)~.
  \]
  In particular, if $Y$ has no free faces and $f$ is an immersion then
  $\chi(Y)\leq\chi(\hatt Y)$.
\end{corollary}

\begin{proof}
  By~(\ref{cells}),
  \[
    \chi(Y)+\sum_i(n_i-1)=\chi(\hgraph)+\deg(\sigma)~,
  \]
  so if $\chi(Y)+\sum_i(n_i-1)>\chi(\hatt Y)$ then
  $\chi(\hgraph)+\deg(\sigma)>\chi({\und})+1$ and by the dependence
  theorem for each edge $e$ of of $\und$, $\vert\partial {\adj}\cap
  {\adj}_e\vert\geq 2$. The map $\partial\adj\to w(E_S)$ is therefore
  at least two-to one, and since $\partial Y=\partial\real{\adj}$, so
  is $f\vert_{\partial Y}$.  This proves the first inequality.
  
The second inequality is clear since the one-skeleton ${\folded}$ of
$\hatt Y^I$ is obtained from the one-skeleton $\und$ of $\hatt Y$ by
folding.
\end{proof}

\section{Proof of the dependence theorem}

\subsection{Stackings}

As well as the adjunction space, the second tool that we will use is
the notion of a \emph{stacking} from \cite{louder-wilton}.  In that
paper, a stacking of a map $\real{w}\colon\real{S}\to\real{\ambient}$
was defined to be a lift of $w$ to an embedding into
$\real{\ambient}\times\mathbb{R}$ (where $\mathbb{R}$ denotes the real
numbers). Here, we use an equivalent, combinatorial, version of the
definition. Given an injection of sets $\alpha\colon C\to D$ and a
total order $\leq$ on $D$, we let $\alpha^*(\leq)$ denote the pullback
order on $C$.

  \begin{definition}[Stacking]
  \label{def:stacking}
  Let $w\colon S\immerses {\ambient}$ be an immersion of graphs. A
  \emph{stacking} of $w$ is a collection of orders $\leq_x$ on
  $w^{-1}(x)$ for $x\in w(S)$, such that
  $\alpha^*(\leq_{\alpha(e)})=\leq_e$ for each $e\in w(E_S)$ and
  $\alpha=\iota$ or $\alpha=\tau$.
\end{definition}

\begin{figure}[ht]
  \centering
    \includegraphics[width=.6\textwidth]{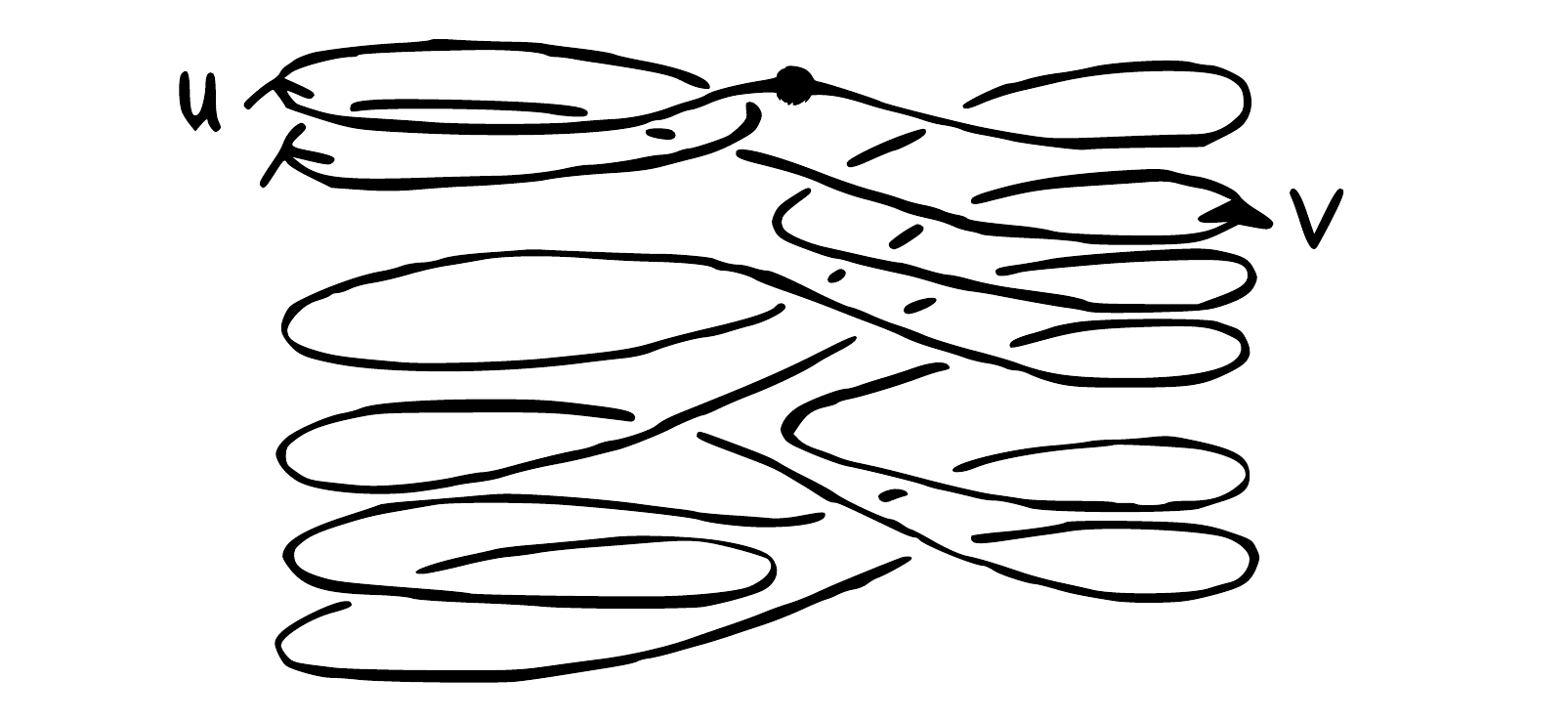}
  \caption{A stacking gives an inclusion $\widetilde{\real{w}}\colon
    \real{S}\into \real{\ambient}\times\mathbb{R}$ and
    vice-versa. This is a picture of a stacking of (the realization
    of) the word $w=uuvuvvUUVUVV$ in the rose with two petals. This
    word can be written as a commutator in two inequivalent ways
    (see~\cite{bestvina-feighn-counting}).}
  \label{stackingfigure}
\end{figure}

\begin{lemma}[{Loo-roll lemma~\cite[Lemma~17]{louder-wilton}}]
  \label{stackings}
  Any indivisible immersion $w\colon S\immerses {\ambient}$ from a
  cycle to a graph has a stacking.
\end{lemma}

For the rest of the paper we will write realizations in normal rather
than boldface font.

\subsection{Computing the characteristic of ${\adj}$}

\label{subsection: computing the characteristic}

In this subsection, we observe that Theorem \ref{maintheorem} can be
proved by estimating the Euler characteristic of a certain chain
complex $\mathcal{C}$ naturally associated to any graph of graphs
$W$. All coefficients are in a fixed but arbitrary field.

The chain complex $\mathcal{C}$ is  
 \[
  0\to H_1(W_E)=\bigoplus_{e\in E_{\Gamma_W}}H_1(W_e)\stackrel{\partial}{\rightarrow} H_1(W_V)=\bigoplus_{v\in V_{\Gamma_W}}H_1(W_v)\to 0\,,
\]
with boundary map $\partial=\tau_{\#}-\iota_{\#}$. As usual, the Euler
characteristic of a chain complex is the alternating sum of the
dimensions of its terms, so
\[
\chi(\mathcal{C})=b_1(W_V)-b_1(W_E)\,.
\]

\begin{lemma}\label{lem: Chain complex for graph of graphs}
  Let $W$ be any graph of graphs with connected vertex and edge
  graphs, and with underlying graph $\Gamma_W$. Then
  \[
  \chi({\adj})+\chi(\mathcal{C})=\chi({\und})\,.
  \]
\end{lemma}
\begin{proof}
  By definition,
  \begin{eqnarray*}
    \chi(W)&=&\chi(W_V)-\chi(W_E)\\
    &=&b_0(W_V)-b_1(W_V)-b_0(W_E)+b_1(W_E)\\
    &=&-(b_1(W_V)-b_1(W_E))+(b_0(W_V)-b_0(W_E))\,.
  \end{eqnarray*}
  Since the vertex and edge spaces are connected,
  $\chi(\Gamma_W)=b_0(W_E)-b_0(W_V)$, and the result follows.
\end{proof}

When $W$ is the adjunction space associated to a branched map,
we obtain that estimating $\chi(\mathcal{C})$ suffices to prove
Theorem \ref{maintheorem}.

\begin{lemma}\label{lem: Chain complex for adjunction space}
  Let $h$ be a branched map of combinatorial complexes as above,
  and let $W$ be the horizontal graph of spaces for the adjunction
  space. Then
  \[
  \chi(\Gamma)+\chi(\mathcal{C})=\chi(\und)\,.
  \]
\end{lemma}
\begin{proof}
This is immediate from the previous lemma, because
\[
\chi(W)=\chi(\Gamma)+\chi(S)-\chi(P)\,,
\]
but $S$ and $P$ are disjoint unions of circles, and so
$\chi(S)=\chi(P)=0$.
\end{proof}

\subsection{Fiberwise filtering ${\adj}$}

\label{fiberwisefiltering}

Let ${\adj}$ be the horizontal graph-of-graphs decomposition for the
adjunction space of a branched map, and consider the chain
complex $\mathcal{C}$ indexed by the graph ${\und}$. In this section
we use stackings to replace $\mathcal{C}$ by a pair of chain complexes
$\mathcal{C}^{\pm}$ indexed by $S$ and which have easily computable
characteristic.

Let
${\adj}_{\subscript}=(S_{\subscript}\sqcup\hgraph_{\subscript},{\pb}_{\subscript},\lambda,\sigma)$
be a (bipartite) vertex or edge graph of ${\adj}$, where
$S_{\subscript}={\adj}_{\subscript}\cap S$,
$\hgraph_{\subscript}={\adj}_{\subscript}\cap \hgraph$,
${\pb}_{\subscript}={\adj}_{\subscript}\cap {\pb}.$ For each vertex
$s\in S_{\subscript}$, let $\pb_s=\sigma^{-1}(s)$.

Suppose that $w\colon S\to {\ambient}$ has a stacking, which we pull
back to a stacking of $w\colon S\to {\und}$. For $s\in S_{\subscript}$
define
\[
  {\adj}_{\subscript}^+(s)=\hgraph_{\subscript}\cup\{t\mid t\leq_{\subscript} s\}\cup\{p\mid\sigma(p)\leq_{\subscript} s\}
\]
and
\[
  {\adj}_{\subscript}^-(s)=\hgraph_{\subscript}\cup\{t\mid s\leq_{\subscript} t\}\cup\{p\mid s\leq_{\subscript} \sigma(p) \}~.
\]

Let $s+1$ be the successor of $s$ and $s-1$ be the predecessor of $s$,
when defined, and interpret ${\adj}_{\subscript}^+(s-1)$ as
$\hgraph_{\subscript}$ if $s$ is minimal and
${\adj}_{\subscript}^-(s+1)$ as $\hgraph_{\subscript}$ if $s$ is
maximal. The order $\leq_{\subscript}$ gives two filtrations of
${\adj}_{\subscript}$ by the sublevel sets
${\adj}_{\subscript}^{\pm}(s)$.
\begin{align}
    \hgraph_{\subscript}\subsetneq\dotsb\subsetneq {\adj}^{+}_{\subscript}(s-1)\subsetneq {\adj}^{+}_{\subscript}(s)\subsetneq
    {\adj}^{+}_{\subscript}(s+1)\subsetneq\dotsb\subsetneq {\adj}_{\subscript} \label{upfiltration}
\end{align}
and
\begin{align}
\hgraph_{\subscript}\subsetneq \dotsb \subsetneq {\adj}^{-}_{\subscript}(s+1)\subsetneq {\adj}^{-}_{\subscript}(s)\subsetneq
    {\adj}^{-}_{\subscript}(s-1)\subsetneq\dotsb\subsetneq {\adj}_{\subscript}~. \label{downfiltration}
\end{align}

For $s\in S_{\subscript}$, define
\[
A^{\pm}(s)=H_1({\adj}_{\subscript}^{\pm}(s))/H_1({\adj}_{\subscript}^{\pm}(s\mp 1))~.
\]
The quotient group $A^{\pm}(s)$ represents the additional first
homology gained when going from ${\adj}_{\subscript}^{\pm}(s\mp 1)$ to
${\adj}_{\subscript}^{\pm}(s)$. See Figure~\ref{filtration}. Summing
over $s\in S_{\subscript}$, we have
\begin{align}
  H_1({\adj}_{\subscript})\cong\bigoplus_{s\in S_{\subscript}}A^{\pm}(s)~.\label{filterdecomposition}
\end{align}
The attaching map $\alpha_e\colon {\adj}_e\to {\adj}_{\alpha(e)}$ is
injective on $S$--vertices and respects the orders $\leq_*$, so
$\alpha^*(\leq_{\alpha(e)})=\leq_e$, and there are therefore
restrictions
\[
  {\adj}_e^{\pm}(s)\to {\adj}_{\alpha(e)}^{\pm}(\alpha(s))
\]
such that
\[
\alpha_e({\adj}_e^{\pm}(s\mp 1))\subseteq {\adj}_{\alpha(e)}^{\pm}(\alpha(s\mp
1))\subseteq {\adj}_{\alpha(e)}^{\pm}(\alpha(s)\mp 1)~.
\]
Because $h$ is a branched map, by Lemma~\ref{lem: branched
  maps and pullbacks}, each $\alpha_e\colon {\pb}_e\to
{\pb}_{\alpha(e)}$ is injective, so $\alpha_s\colon {\pb}_s\to
{\pb}_{\alpha(s)}$ is as well, so there are induced injections
\begin{align}
  \alpha^{\pm}_{s,\#}\colon A^{\pm}(s)\into A^{\pm}(\alpha(s))~.\label{boundaryonelevel}
\end{align}
Again, summing over $s\in S_{\subscript}$, there are maps
\begin{align}
\alpha^{\pm}_{e,\#}=\bigoplus_{s\in S_e}\alpha^{\pm}_{s,\#}\colon\bigoplus_{s\in S_e}A^{\pm}(s)\into\bigoplus_{s\in S_{\alpha(e)}}A^{\pm}(s)~. \label{bipartiteinclusion}
\end{align}

\begin{figure}[ht]
  \centerline{
    \includegraphics[width=.6\textwidth]{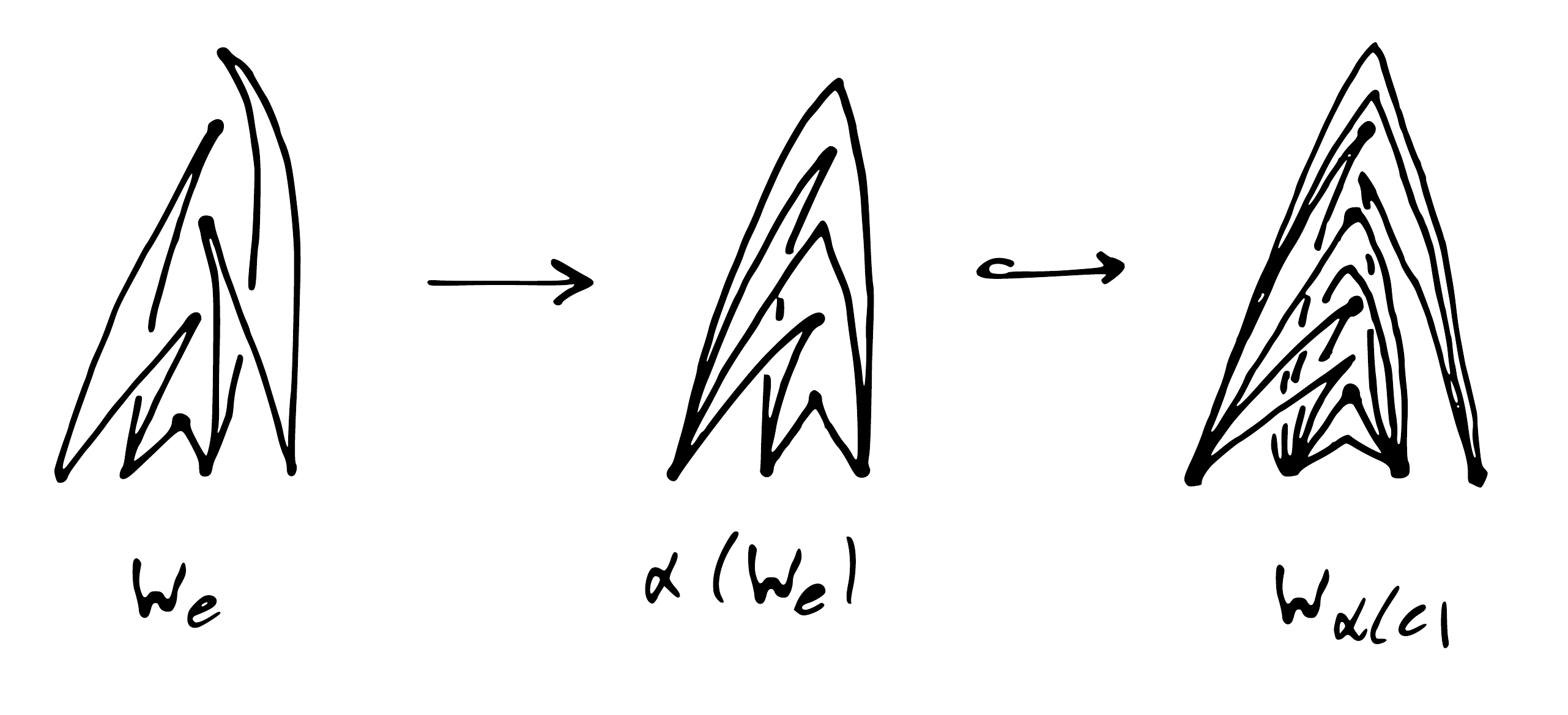}
  }
  \caption{The map $\alpha\colon {\adj}_e\to {\adj}_{\alpha(e)}$ is
    injective on ${\pb}_e$ and induces an injection $A^{\pm}(s)\into
    A^{\pm}(\alpha(s)).$ In this example two vertices of $\hgraph_e$
    are identified in $\hgraph_{\alpha(e)}$. The map $\alpha$ respects
    the sublevelset filtrations (\ref{upfiltration}) and
    (\ref{downfiltration}). Here we have drawn $S_{\subscript}$ as
    sitting ``above'' the $\hgraph_{\subscript}$ so this picture
    should be thought of as illustrating the
    filtration~(\ref{upfiltration}).}
  \label{inclusion}
\end{figure}

We now define a pair of auxiliary chain complexes $\mathcal{C}^{\pm}$
by replacing each $H_1({\adj}_{\subscript})$ in $\mathcal{C}$ using
the isomorphism (\ref{filterdecomposition}), using the sum of the maps
from~(\ref{bipartiteinclusion}) as the boundary map, that is
\[
\partial^{\pm}\define\bigoplus_{e\in E_{\und}}\tau^{\pm}_{e,\#}-\iota^{\pm}_{e,\#}~.
\]
And so
\begin{align}
  \mathcal{C}^{\pm}=\left(0\to\bigoplus_{e\in E_{{\und}}}\bigoplus_{s\in S_e}A^{\pm}(s)\stackrel{\partial^{\pm}}{\longrightarrow}\bigoplus_{v\in V_{{\und}}}\bigoplus_{s\in S_v}A^{\pm}(s)\to 0\right)~.\label{fcomplex}
\end{align}
By~(\ref{filterdecomposition}),
$\chi(\mathcal{C}^{\pm})=\chi(\mathcal{C})$. Since
\[
  V_S=\bigsqcup_{v\in V_{{\und}}}S_v\mbox{ and } E_S=\bigsqcup_{e\in E_{{\und}}}S_e~,
\]
after reindexing,~(\ref{fcomplex}) becomes
\[
  \mathcal{C}^{\pm}=\left(0\to \bigoplus_{e\in E_S}A^{\pm}(e) \to \bigoplus_{v\in V_S}A^{\pm}(v)\to 0\right)~,
\]
with boundary maps coming from~(\ref{boundaryonelevel}).

These auxiliary chain complexes enable us to relate
$\chi(\mathcal{C})$ to the vector spaces $A^{\pm}(s)$ that come from
the filtrations of the ${\adj}_{\subscript}$.

\begin{lemma}
  \label{charc}
  Suppose $S$ is a cycle. Then
  \[
  \max\{\dim(A^{\pm}(s))\mid s\in S\}\leq \chi(\mathcal{C})~.
  \]
\end{lemma}

The proof uses the following naive estimate.

\begin{remark}\label{rem: Naive estimate}
  Let $a_1,\dotsc,a_n$ and $b_1,\dotsc,b_{n-1}$ be non-negative
  integers, and suppose that $a_i\geq b_i\leq a_{i+1}$ for $i=1\dotsc
  n-1$. Then
  \[
  a_1-b_1+\dotsb-b_{n-1}+a_n\geq \max\{a_i,b_i\}~.
  \]
\end{remark}

\begin{proof}[Proof of Lemma~\ref{charc}]
  Pick an edge $g\in w(E_S)\subseteq E_{\und}$, and let $m^+$ and
  $m^-$ be the minimal and maximal elements of $S_g$ with respect to
  the order $\leq_g$. Since $m^{\pm}$ is minimal/maximal,
  \[
    V_{\adj_g^{\pm}(m^{\pm})}=\hgraph_g\cup\{m^{\pm}\}
  \]
  and
  \[
    E_{\adj_g^{\pm}(m^{\pm})}=\{p\mid\sigma(p)=m^{\pm}\}~.
  \]
  By Lemma~\ref{lem: DI properties} ${\adj}_g$ is simple, so if $p\in
  E_{\adj_g^{\pm}(m^{\pm})}$ then $p$ is determined by $\lambda(p)$,
  and ${\adj}^{\pm}_g(m^{\pm})$ is therefore $\hgraph_g$ with
  $\lambda({\pb}_{m^{\pm}})$ coned off, so $A^+(m^+)\cong
  A^-(m^-)\cong 0$.  Removing $m^{\pm}$ from $S$ therefore doesn't
  change the characteristic of the chain complexes
  $\mathcal{C}^{\pm}$, i.e.
  \[
    \chi(\mathcal{C}^{\pm})=\chi(\mathcal{C}^{\pm}\vert_{S\smallsetminus{m^{\pm}}})
  \]
  where
  \[
    \mathcal{C}^{\pm}\vert_{S\smallsetminus m^{\pm}}=\left(0\to \bigoplus_{e\in
        E_S\smallsetminus m^{\pm}}A^{\pm}(e) \to \bigoplus_{v\in
        V_S}A^{\pm}(v)\to 0\right).
  \]
  The chain complex $\mathcal{C}^{\pm}\vert_{S\smallsetminus m^{\pm}}$
  is over an interval $S\smallsetminus m^{\pm}$, which makes its Euler
  characteristic easy to estimate. Label and reorient $S$ so that
  $V_S=\{v^{\pm}_1,\dotsc,v^{\pm}_n\}$ and
  $E_S=\{m^{\pm},e^{\pm}_1,\dotsc,e^{\pm}_{n-1}\}$ with
  $\iota(e^{\pm}_i)=v^{\pm}_i$ (for $i=1,\dotsc, n$) and
  $\tau(e^{\pm}_i)=v^{\pm}_{i+1}$ (for $i=1,\dotsc, n-1$). Set
  $a^{\pm}_i=\dim(A^{\pm}(v^{\pm}_i))$ and
  $b^{\pm}_j=\dim(A^{\pm}(e^{\pm}_j))$. Then
  \[
  \chi(\mathcal{C})=\chi(\mathcal{C}^{\pm})=a^{\pm}_1-b^{\pm}_1+a^{\pm}_2-\dotsb+a^{\pm}_{n-1}-b^{\pm}_{n-1}+a^{\pm}_n~.
  \]
  Since $\alpha\colon A^{\pm}(e)\to A^{\pm}(\alpha(e))$ is injective,
  $a^{\pm}_i\geq b^{\pm}_i\leq a^{\pm}_{i+1}$ for $i=1,\dotsc ,n-1$, and
  \[
  \chi(\mathcal{C})\geq\max\{a^{\pm}_i,b^{\pm}_i\}=\max\{\dim(A^{\pm}(s))\mid
  s\in S\}\geq 0
  \]
  by Remark \ref{rem: Naive estimate}.
\end{proof}

\begin{remark}
  It is not clear from the start that $\chi(\mathcal{C})$ is
  non-negative. It follows from Mayer--Vietoris that the chain
  complexes $\mathcal{C}^{\pm}\smallsetminus m^{\pm}$, and therefore
  $\mathcal{C}^{\pm}$, have their homology concentrated in dimension
  $0$.
  \[
  \chi(\mathcal{C}^{\pm})=\dim(H_0(\mathcal{C}^{\pm}))
  \]

  The special case $\chi(\hgraph)=\chi({\und})$ is of some interest
  since it implies the theorems of Baumslag and Stallings.  In these
  cases $\chi(\mathcal{C})=0$, and by Lemma~\ref{charc}
  $\dim(A^{\pm}(s))=0$ for all $s\in
  S$. By~(\ref{filterdecomposition}), $H_1({\adj}_{\subscript})=0$ for
  all $\subscript\in{\und}$, but a connected graph with trivial
  homology is a tree. If $\deg(\sigma)\geq 2$ then no $s\in
  S_{\subscript}$ has valence one, so there are at least two
  valence-one vertices in $\hgraph_{\subscript}$, hence $\lambda$ is
  strongly reducible, and therefore reducible. This case is argued
  differently in the paper~\cite{adjoiningroots}. There it was shown
  directly that the vertices in $\hgraph_{\subscript}$ are cutpoints
  in $\adj_{\subscript}$, and acylindricity of the associated graph of
  groups $\Delta$ then implied that the edge and vertex spaces are
  trees. Since this is not true in general, we use stackings to argue
  indirectly that if $\chi(\mathcal{C})<\deg(\sigma)-1$ then the edge
  spaces have ``treelike'' features, and ultimately, valence one
  vertices.
  
\end{remark}

\subsection{The up-down lemma and the proof of Theorem~\ref{maintheorem}}\label{Ss: Updown lemma}

The final ingredient of the proof of the dependence theorem is the
\emph{up-down lemma}. To formulate it, we first recapitulate some of
the discussion from Section~\ref{fiberwisefiltering} in general terms.

Consider a finite bipartite graph $B=(U\sqcup V,E,\sigma,\lambda)$
with an order $\leq$ on $V$. For $v\in V$ define
\[
B^+(v)=U\cup\{v'\mid v'\leq v\}\cup\{e\mid\sigma(e)\leq v\}
\]
and
\[
B^-(v)=U\cup\{v'\mid v'\geq v\}\cup\{e\mid\sigma(e)\geq v\}\,.
\]
Let
\[
A^{\pm}(v)=H_1(B^{\pm}(v))/H_1(B^{\pm}(v\mp 1))\,,
\]
where we interpret $B^+(v-1)$ as $U$ if $v$ is minimal and $B^-(v+1)$
as $U$ if $v$ is maximal. A vertex $v\in V$ is \emph{good} if
\[
\max\{\dim(A^{\pm}(v))\}=\valence(v)-1~.
\]
A vertex $u\in U$ is \emph{good} if it has valence one.

\begin{figure}[ht]
  \centerline{
    \includegraphics[width=\textwidth]{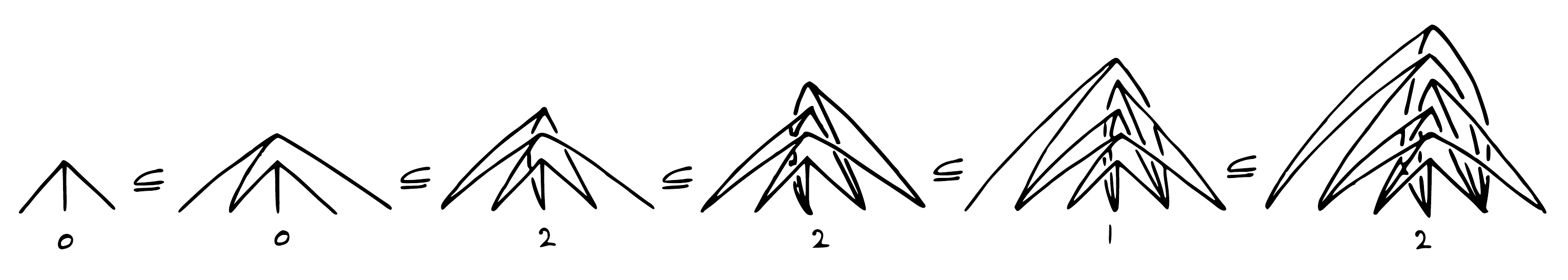}
  }
  \caption{Illustration of a filtration associated to an order $\leq$
    on a (simple) bipartite graph $B$. The elements of $U$ are all
    drawn at the same level, and elements of $V$ are placed
    vertically. To keep the pictures uncluttered we omit elements of
    $U$ which aren't connected to vertices in $V\cap B^+(v)$. The
    number below each graph is the dimension of $A^+(v)$ for the
    vertex $v$ added at that stage. The graph $B$ has $6+6$ vertices
    and $18$ edges, for a characteristic of $-6$, and is connected
    with first betti number $0+0+2+2+1+2$.}
  \label{filtration}
\end{figure}

\begin{lemma}[Up-down lemma]
  \label{updownlemma}
  Let $B$ be a simple connected bipartite graph which is not a
  point. Let $\leq$ be an order on $V$. Then
  \[
    \vert\{p\in V\cup U\mid p\hfill\mbox{ is good.}\}\vert  \geq 2\,.
  \]
\end{lemma}

\def\maximum{{m^-}}
\def\minimum{{m^+}}

\begin{proof}
  The proof is by induction on $\vert V\vert$. Suppose that $\vert
  V\vert=1$. If $\vert U \vert=1$ then $V=\{v\}$, $U=\{u\}$, $v$ has
  valence $1$, $\dim(A^{\pm}(v))=\valence(v)-1=0$, so $v$ is good, and
  $\vert\lambda^{-1}(u)\vert=1$ so $u$ has valence one, so is good. If
  $\vert U\vert\geq 2$ then there are $\vert U\vert\geq 2$ valence one
  vertices in $U$.

  Suppose that $\vert V\vert\geq 2$, and let $\maximum$ and $\minimum$
  be the maximal and minimal elements of $V$, respectively. If
  $\maximum$ and $\minimum$ are both good then we are done.

  The long exact sequence for the pair $(B,B^+(\maximum-1))$
  reduces to the exact sequence
  \begin{equation}
    0\to A^+(\maximum) \to H_1(B,B^+(\maximum-1))\to
    H_0(B^+(\maximum-1))\to H_0(B)\to 0\,.\label{les}
  \end{equation}
  Since $B\smallsetminus B^+(\maximum-1)$ has one vertex $\maximum$
  and has $\valence(\maximum)$ edges connecting $B^+(\maximum-1)$ to
  $\maximum$, the relative homology group $H_1(B,B^+(\maximum-1))$ is
  $\valence(\maximum)-1$ dimensional. Since $B$ is connected,
  $\dim(H_0(B))=1$.  Suppose now that $\maximum$ is not good. Since
  $B$ is simple, $\dim(A^-(\maximum))=0$, and since $\maximum$ is not
  good, $\dim(A^+(\maximum))<\valence(\maximum)-1$, so by~(\ref{les})
  $\dim(H_0(B^+(\maximum-1)))>1$, and $B^+(\maximum-1)$ is therefore
  not connected and $B\smallsetminus \maximum$ has at least two
  connected components.  Let $B_\maximum$ be the closure of a
  connected component of $B\smallsetminus \maximum$ which doesn't
  contain $\minimum$. By induction on $\vert V\vert$, $B_\maximum$ has
  at least two good vertices, one of which is not $\maximum$. Let $g$
  be this vertex. If $\minimum$ is good then $\minimum$ and $g$ are
  both good.  Argue symmetrically if $\maximum$ is good and $\minimum$
  is not good.

  Thus we assume both $\maximum$ and $\minimum$ are not good. Again,
  let $B_\maximum$ be the closure of a connected component of
  $B\smallsetminus \maximum$ which doesn't contain $\minimum$, and let
  $B_\minimum$ be the closure of a connected component of
  $B\smallsetminus \minimum$ which doesn't contain $\maximum$. The
  vertices $\maximum$ and $\minimum$ are good in $B_\maximum$ and
  $B_\minimum$, respectively, and $B_\maximum$ and $B_\minimum$ are
  disjoint. By induction on $\vert V\vert$, $B_\maximum$ and
  $B_\minimum$ each contain at least two good vertices, at least one
  of which is not $\maximum$ or $\minimum$, respectively. A good
  vertex in $B_\maximum$ which is not $\maximum$ is good in $B$, and a
  good vertex in $B_\minimum$ which is not $\minimum$ is good in $B$
  as well, so $B$ has at least two good vertices.
\end{proof}

\begin{figure}[ht]
  \centerline{
    \includegraphics[width=.9\textwidth]{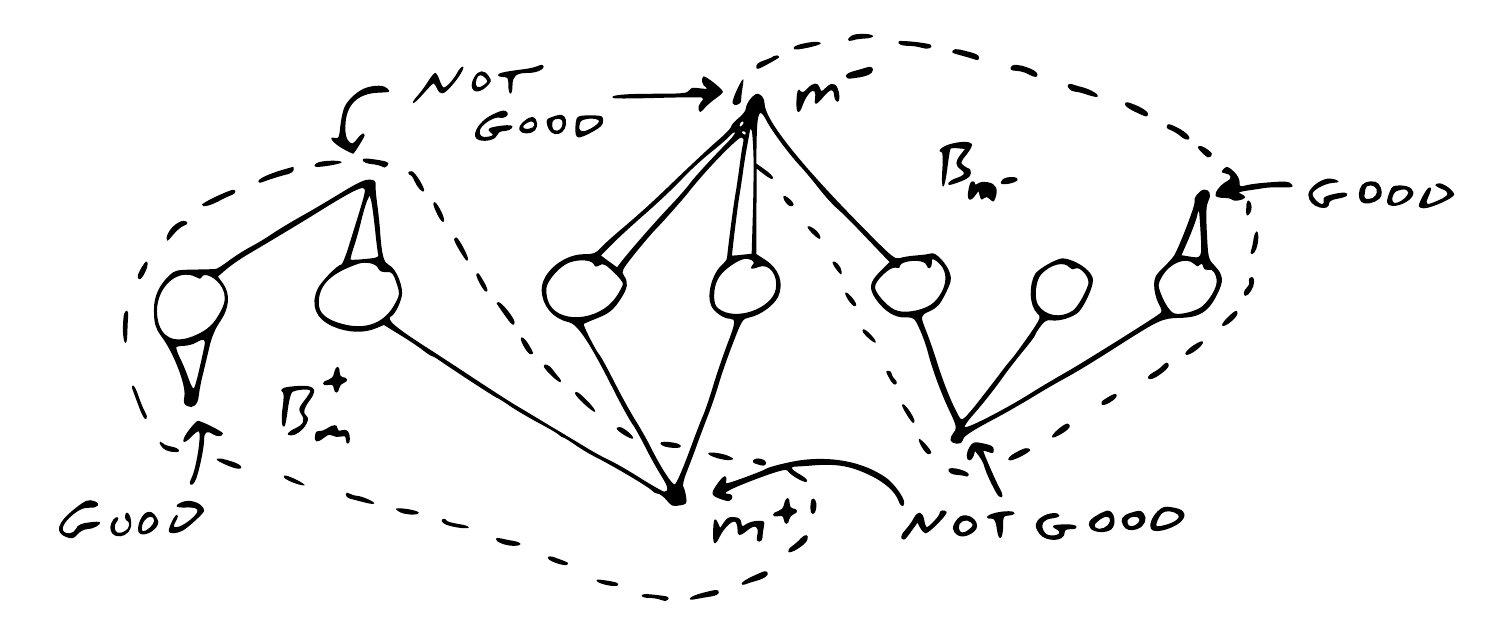}
  }
  \caption{Illustration for Lemma~\ref{updownlemma}. In this case
    neither $\maximum$ nor $\minimum$ is good. We picture $U$ as
    sitting below $\maximum$ and above $\minimum$.}
\end{figure}

With the up-down lemma in hand, we can finally prove the dependence
theorem.

\begin{proof}[Proof of Theorem~\ref{maintheorem}]
  We prove the contrapositive. Suppose that
  \[
    \chi(\hgraph)+\deg(\sigma)-1>\chi({\und})~.
  \]
  Our goal is to prove that ${\adj}$ is strongly independent.
    
  By Lemma~\ref{charc}, $\chi(\mathcal{C})$ is bounded from below by
  \[
    \max_{\subscript\in {\und}}\max_{s\in S_{\subscript}}\{\dim(A^{\pm}(s))\}
  \]
  so
  \[
    \chi(\hgraph) +
    \max_{\subscript\in {\und}}\max_{s\in S_{\subscript}}
    \{\dim(A^{\pm}(s))\}\leq
    \chi(\hgraph)+\chi(\mathcal{C}) = \chi({\und})\,,
  \]
  where the last equality is given by Lemma \ref{lem: Chain complex
    for adjunction space}.  If
  $\chi(\hgraph)+\deg(\sigma)-1>\chi({\und})$ then
  \begin{align}
    \max_{\subscript\in {\und}}\max_{s\in S_{\subscript}}\{\dim(A^{\pm}(s))\}<\deg(\sigma)-1~.\label{nearlythere}
  \end{align}
  To show that ${\adj}$ is strongly independent, we need to show that
  $\vert\partial {\adj}\cap {\adj}_e\vert\geq 2$ for each $e\in
  w(E_S)\subseteq E_{\und}$.
  
  To that end, choose $e\in w(E_S)$. By Lemma~\ref{lem: DI
    properties}(i), $W_e$ is a simple graph, so we may apply the
  up-down lemma to ${\adj}_e$ by setting $B={\adj}_e$, $V=S_e$,
  $U=\hgraph_e$, $E={\pb}_e$, and $\leq=\leq_e$. Remark~\ref{rem:
    Cycle valence} asserts that $\deg(\sigma)=\valence(s)$,
  so~(\ref{nearlythere}) implies that $\dim(A^{\pm}(s))<\valence(s)-1$
  for all $s\in S_e$.  In particular, no vertex in $S_e$ is
  good. Since the up-down lemma guarantees two good vertices in
  ${\adj}_e$, it follows that there are two good vertices in
  $\hgraph_e$. A good vertex in $\hgraph_e$ has valence one, so
  $\vert\partial {\adj}\cap {\adj}_e\vert\geq 2$.
  
  This is true for all $e\in w(E_S)$, but this is precisely what it
  means for the map $\lambda\colon {\pb}\to \hgraph$ to be strongly
  independent.
\end{proof}

\section{Stallings; Magnus and Lyndon; Duncan--Howie}
\label{applications}

In this section we show how the dependence theorem implies its
predecessors mentioned in the introduction.  We have already seen that
it implies Theorem \ref{introthm: main}, which in turn implies
Baumslag's theorem.  We next state a generalization of Stallings'
theorem and explain how it follows as well.  In the following
subsection we explain how the dependence theorem implies Magnus'
Freiheitssatz and Lyndon's asphericity theorem. Finally, we explain
how the dependence theorem implies a strengthening of the theorem of
Duncan--Howie.

\subsection{Conjugacy and homology}

\label{subsec: stallings}

A homomorphism of free groups $f\colon\HH\to{\FF}$ induces a map
$f_{\sim}\colon
\HH/\negthinspace\negthinspace\sim\thinspace\to{\FF}/\negthinspace\negthinspace\sim$
on sets of conjugacy classes.  A 1983 theorem of Stallings, which we
also think of as a kind of dependence theorem, relates $f_{\sim}$ to
the induced map on abelianizations, $f_{\#}\colon H_1(\HH)\to
H_1({\FF})$ \cite[Theorem 5.3]{stallings-surfaces}.

\begin{thm}[Stallings]
  Let $f\colon\HH\to{\FF}$ be an injection of finitely generated free
  groups. If $f_{\#}$ is injective then so is $f_\sim$.
\end{thm}
\noindent(A homomorphism $f$ for which $f_\sim$ is injective is
sometimes called a \emph{Frattini embedding}; cf.\
\cite{olshanskii_conjugacy_2004}.)

In this section we quantify Stallings' theorem, and compare how badly
$f_\sim$ and $f_\#$ may fail to be injective. In the case of $f_\#$,
the failure of injectivity is measured by the rank of the kernel. To
measure the failure of $f_\sim$ to be injective, we define
\[
  \gamma(f)=\max_{\left[v\right]\in{\FF}/\negthinspace\sim}\{\vert f^{-1}_{\sim}(\left[ v\right])\vert\}\in\NN\cup\{\infty\}~,
\]
the maximal number of conjugacy classes in $\HH$ that are identified
in ${\FF}$.  Using this terminology, Stallings' theorem asserts: if
$\gamma(f)>1$ then $\rk{\ker(f_\#)}>0$.
 
The main result of this section is a corollary of the dependence
theorem that strengthens Stallings' theorem by comparing
$\rk{\ker(f_\#)}$ to $\gamma(f)$.

\begin{corollary}
  \label{stallings}
  Let $f\colon\HH\to{\FF}$ be an injection of finitely generated free
  groups. Then
  \[
    \rk{\ker(f_{\#})}\geq\gamma(f)-1\,.
  \]
\end{corollary}
\begin{proof}
  By Marshall Hall's theorem, $f$ identifies $H$ with a free factor in
  some subgroup $K$ of finite index in $F$. Each conjugacy class of
  $F$ splits into at most $|F:K|$ conjugacy classes in $K$, and does
  not split any further in $f(H)$. Therefore, $\gamma(f)\leq |F:K|$,
  and in particular is finite.

  The proof proceeds by induction on $m=\gamma(f)$. In the base case,
  $m=1$, there is nothing to prove, so we assume that $m\geq 2$.  We
  may also assume that $\HH$ and ${\FF}$ are finitely generated.  Let
  $u_1,\dotsc,u_m$ be a collection of non-conjugate elements realizing
  $\gamma(f)$.
  
  Since free groups have unique roots, for each $u_j$ there is a
  unique $v_j\in \HH$ such that $u_j= v_j^{k_j}$, with $k_j\geq 1$
  maximal.  Using uniqueness of roots again, it follows that
  $\{\langle v_j\rangle\}$ forms a malnormal family of cyclic
  subgroups of $\HH$.  Since the $f(u_j)$ are all conjugate to each
  other, the $f(v_j)$ are all conjugate into some common cyclic
  subgroup $\langle w\rangle$ of $\FF$. Therefore, each $f(v_j)$ is
  conjugate to $w^{n_j}$ for some unique integer $n_j$.  As in the
  introduction, these data define a graph of groups $\Delta$ and $f$
  extends to a homomorphism $\phi\colon \pio{\Delta}\to{\FF}$.  Let
  $L=\phi(\pio{\Delta})\leq{\FF}$.
   
  Since $f(\HH)$ is contained in $L$ we have
  $\rk{\im(f_\#)}\leq\rk{L}$ and so the rank-nullity lemma applied to
  $f_\#$ gives
  \[
  \rk{\ker(f_\#)}= \rk{\HH}-\rk{\im(f_\#)}\geq \rk{\HH}-\rk{L}~.
  \]
  If the malnormal family $\{\langle v_j\rangle\}$ is dependent then
  Theorem \ref{introthm: main} implies that
  \[
    \rk{\HH}-\rk{L} \geq \left(\sum_{i=1}^m n_i\right)-1\geq m-1~.
  \]
  These two estimates together imply the result, so it remains to deal
  with the case in which $\{\langle v_j\rangle\}$ is independent.

  After permuting indices and conjugating the $v_j$ appropriately,
  this means that
  \[
  \HH=\KK*\langle v_m\rangle
  \]
  and $v_j\in \KK$ for $j<m$.  Therefore, by the inductive hypothesis
  applied to $f\vert_{\KK}\colon \KK\to{\FF}$, we have
  $\rk{\ker(f_{\#}\vert_{H_1(\KK)})}\geq m-2$.  Since $f(v_1)^{n_m}$
  is conjugate to $f(v_m)^{n_1}$, the class
  \[
  c = n_m [v_1] - n_1 [v_m]
  \]
  is non-zero in $H_1(\HH)$, is contained in the kernel of $f_{\#}$,
  but is not in $H_1(\KK)$. Therefore,
  \[
  \rk{\ker(f_{\#})}\geq\rk{\ker(f_{\#}\vert_{H_1(\KK)})}+1\geq m-1
  \]
  as required.
  \end{proof}

\begin{remark}
Corollary~\ref{stallings} is sharp. Let
${\FF}=\langle a,b\rangle$ and
\[
 \HH=\langle a,bab^{-1},\dotsc,b^{n-2}ab^{2-n},b^{n-1}ab^{1-n}\rangle
\]
with $f$ the inclusion map. The $n$ basis elements $b^iab^{-i}$ of
$\HH$ are conjugate in ${\FF}$ and $\rk{\ker(f_{\#})}=n-1$.
\end{remark}

\subsection{The Freiheitssatz and Lyndon asphericity}

We again consider a one-relator group $\GG={\FF}/\ncl{w}$. As usual,
we think of $\FF$ as the fundamental group of a graph $\ambient$, and
realise $w$ as an immersion $S\immerses \ambient$, where $S$ is a
cycle. Note that $w$ may be a proper power $v^k$, where $k\geq 1$ is
assumed to be maximal. In this section we show how
Corollary~\ref{monotonicity} implies the Freiheitssatz and Lyndon
asphericity. In what follows $X$ is the presentation complex
$({\ambient},S,w)$ of the one-relator group $G$, where
$w\colon{S}\to{\ambient}$ is the attaching map of the two cell, and
$Z$ is the presentation complex of the one-relator group
$({\ambient},S,v)$. There is a natural map $q\colon X\to Z$, equal to
the identity on $\Omega$ and a $k$ sheeted cover on $S$. Note that $q$
is \emph{not} a branched map in the sense of Definition
\ref{def: branched map} if $k>1$.

\begin{definition}[Surface diagram]
  A \emph{singular surface diagram} in $X$ is a morphism $f\colon Y\to
  X$ of combinatorial complexes, such that the link of every vertex in
  $Y$ is a union of points, cycles and intervals. A singular surface
  diagram $f\colon Y\to X$ is \emph{reduced} if the induced map
  $q\circ f$ is a branched map.
\end{definition}

This definition agrees with the usual notions of reduced disk and
sphere diagram.  The following theorem, which is the main theorem of
this section, is a common generalization of Magnus' Freiheitssatz and
Lyndon asphericity.

\begin{theorem}[Magnus, Lyndon]\label{thm: Magnus Lyndon}
  Let $X$ be the presentation complex of a one-relator group, and
  $f\colon Y\to X$ a reduced singular surface diagram. If $\chi(Y)\geq
  1$ then $w(S)\subseteq f(\partial Y)$.
\end{theorem}
\begin{proof}
  If $w(S)\not\subseteq f(\partial Y)$ then certainly
  $w(S)\not\subseteq q(f(\partial Y))$, so we may replace $X$ by $Z$.
  Let $\hatt Y$ be the one-relator pushout of the map $Y\to Z$. By
  Corollary~\ref{monotonicity}, $\chi(Y)\leq\chi(\hatt Y)$, so
  $\chi(\hatt Y)\geq 1$. Since $\hatt Y$ is one-relator, and $v$ is
  indivisible, $\hatt Y$ is the disk $D$, and $\und=\partial D$ is a
  cycle. Since $Y\to Z$ is a branched map, it doesn't fold
  faces, but since $Y\to Z$ factors through $D$, no two two-cells in
  $Y$ share an edge. Thus $Y$ is a tree of disks. In this case,
  $\partial Y$ clearly surjects $w(S)$.
\end{proof}

Magnus' Freiheitssatz \cite{magnus-freiheitssatz}, corresponding to
the case when $Y$ is a disk, and Lyndon asphericity
\cite{lyndon-cohomology,cockcroft_two-dimensional_1954}, corresponding
to the case where $Y$ is a sphere, follow immediately.

\begin{corollary}[Magnus' Freiheitssatz]\label{cor: Freiheitssatz}
  Consider a one-relator group $G=F/\ncl{w}$.  If $\HH$ is a proper
  free factor of free group ${\FF}$ and the natural map $\HH\to G$ is
  not injective then $w$ is conjugate into $\HH$.
\end{corollary}
\begin{proof}
  Let $\HH$ be a free factor of $\FF$ and $\gamma \in \HH$ an element
  that dies in $G$.  Take $X$ to be a presentation complex of $G$ in
  such a way that $\HH$ is realized by a subgraph of the one-skeleton.
  We can realize $\gamma$ as an immersed loop in $X$ with image in the
  subgraph that carries $\HH$.  Since $\gamma$ dies in $G$, $\gamma$
  factors through a reduced disk diagram $f:Y\to X$ by van Kampen's
  lemma, in such a way that $\gamma$ surjects the boundary $\partial
  Y$. By Theorem \ref{thm: Magnus Lyndon} the restriction of $f$ to
  $\partial Y$ surjects $w(S)$, and hence $\gamma$ does too. Since
  $\gamma\in\HH$, it follows that $w$ is conjugate into $H$.
\end{proof}

\begin{corollary}[Lyndon asphericity]\label{cor: Lyndon asphericity}
Let $X$ be the presentation complex of a one-relator group
$G=F/\ncl{w}$.  If $Y$ is homeomorphic to a 2-sphere, no combinatorial
map $Y\to X$ is reduced.
\end{corollary}

\subsection{Roots of products of commutators}

\begin{definition}
  Let ${\FF}$ be a free group. The \emph{genus} or \emph{commutator
    length} of an element $v\in{\FF}$ is defined to be the minimal
  $g\in\NN$ such that
  \[
    v=[a_1,b_1]\dotsb[a_g,b_g]~.
  \]
\end{definition}

The Duncan--Howie theorem is an estimate on the commutator length of a
proper power $v=w^n$: it asserts that $n\leq 2g-1$
\cite{duncan-howie}. Here, we view it as a dependence theorem about
maps $\HH\to\FF$ where $\HH$ is the fundamental group of a surface
$\Sigma$ with boundary, and $\partial\Sigma$ maps to powers of
conjugates of $w$.  In this section, we prove another corollary of
Theorem \ref{maintheorem}, which strengthens the Duncan--Howie
theorem.

\begin{corollary}
  \label{dhcorollary}
  Let ${\FF}$ be a free group and consider $v$ a non-trivial element
  which is both a $k$--th power and a product of $g$ commutators, that
  is there are $a_i,b_i,w\in{\FF}$ with $1\leq i\leq g$ and
  \[
    v=[a_1,b_1]\dotsb[a_g,b_g]=w^k~.
  \]
  Then
  \[
    \rk{\langle a_1,\ldots,a_g,b_1,\ldots,b_g,w\rangle} +k-1\leq 2g~.
  \]
\end{corollary}

Since $w\neq 1$, the group $\langle
a_1,\ldots,a_g,b_1,\ldots,b_g,w\rangle$ is a non-abelian free group,
and hence has rank at least 2. Therefore, $k\leq 2g-1$, recovering the
Duncan--Howie estimate.

\begin{figure}[ht]
  \centerline{
    \includegraphics[width=.8\textwidth]{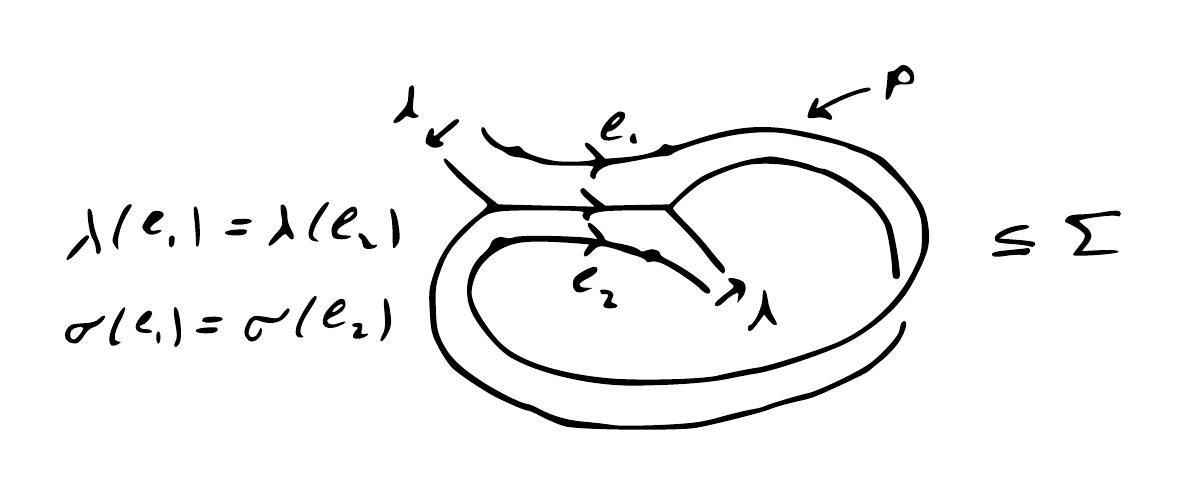}
  }
  \caption{When $\Sigma$ is orientable the map $\lambda\colon {\pb}\to
    \hgraph$ is a branched map since otherwise $\Sigma$ contains
    a M\"obius band.}
  \label{fig: Mobius}
\end{figure}

\begin{proof}[Proof of Corollary~\ref{dhcorollary}.]
  Represent the subgroup $\langle a_i,b_i\rangle\leq{\FF}$ by a map
  $f\colon\Sigma\to {\ambient}$ from an orientable surface of genus
  $g$ with one boundary component, so that $f\vert_{\partial\Sigma}$
  represents the element $v$. We may assume that $f$ doesn't pinch any
  simply closed curves, and that $w$ is indivisible in $\langle
  a_1,\ldots,a_g,b_1,\ldots,b_g,w\rangle$. By \cite{culler-surfaces},
  we may realize $\Sigma$ as the mapping cylinder of $\lambda\colon
  {\pb}\to \hgraph$, where ${\pb}$ is a cycle representing the
  boundary of $\Sigma$, with a morphism of graphs $h\colon \hgraph\to
  {\ambient}$ representing $\langle
  a_1,\ldots,a_g,b_1,\ldots,b_g\rangle$. Orientability of $\Sigma$
  implies that $\lambda$ is a branched map. See Figure~\ref{fig:
    Mobius}.  The induced map from the pushout ${\und}$ surjects
  $\langle a_1,\ldots,a_g,b_1,\ldots,b_g,w\rangle$, and the inequality
  then follows from the dependence theorem.
\end{proof}

\section{Subgroups of one-relator groups}

The results of this section show how $\pi(w)$ controls the subgroup
structure of the one-relator group $\GG=\FF/\ncl{w}$.

\subsection{Primitivity rank and $w$--subgroups}

\label{subsection: primitivity rank and w subgroups}

Recall the definition of the primitivity rank $\pi(w)$ from the
introduction (Definition \ref{defn: Primitivity rank}).  We start with
a few simple observations.

\begin{enumerate}[(i)]
\item The word $w$ is primitive in ${\FF}$ if and only if
  $\pi(w)=\infty$.
\item Unless $w$ is primitive, $\rk{{\FF}}$ is an upper bound for
  $\pi(w)$.
\item The word $w$ is a proper power if and only if $\pi(w)=1$.
\item If $w$ is contained in a subgroup $H\leq \FF$ and $w$ is
  primitive in $\FF$ then $w$ is primitive in $H$. Therefore, the
  primitivity rank of $w$ when viewed as an element of $H$ is at least
  the primitivity rank of $w$ when viewed as an element of $\FF$.
\end{enumerate}

We now turn to the second definition needed for the main lemma.

\begin{definition}\label{defn: w-subgroup}
  Let ${\FF}$ be a free group and $w\in{\FF}$ a non-trivial element. A
  subgroup $\KK$ of ${\FF}$ is a \emph{$w$--subgroup} if:
  \begin{enumerate}[(i)]
  \item $\KK$ contains $w$ as an imprimitive element;
  \item $\rk{\KK}=\pi(w)$; and
  \item every proper overgroup $\KK'$ of $\KK$ in ${\FF}$ has
    $\rk{\KK'}>\rk{\KK}$.
  \end{enumerate}
\end{definition}

In the easiest case $w$--subgroups are cyclic; this occurs if and only
if $\pi(w)=1$, i.e.\ when $w$ is a proper power $u^k$.

\begin{example}\label{ex: w-subgroups of powers}
  If $w=u^k\in{\FF}$ with $k>1$ and $u$ not a proper power then
  $\langle u\rangle$ is the unique $w$--subgroup of ${\FF}$.  It is
  well-known that the inclusion $\langle u\rangle/\ncl{w}\to
  {\FF}/\ncl{w}$ is injective \cite[Proposition
    II.5.17]{lyndon-schupp}.
\end{example}

So when $\pi(w)=1$, a $w$--subgroup is unique and malnormal. In fact,
malnormality holds in general.

\begin{lemma}\label{lem: w-subgroups of free groups are malnormal}
  If $\KK\leq{\FF}$ is a $w$--subgroup then $\KK$ is malnormal.  In
  particular, if $w^g\in\KK$ then $g\in\KK$.
\end{lemma}
\begin{proof}
  Let $g\in{\FF}$; then $\KK\leq\langle K,g\rangle$ and $\rk{\langle
    K,g\rangle}\leq\rk{\KK}+1$.  If $k_1^g=k_2$ for
  $k_1,k_2\in\KK\smallsetminus 1$ then there is a non-trivial relation
  between $\KK$ and $g$ and so, since free groups are Hopfian,
  $\rk{\langle K,g\rangle}\leq\rk{\KK}$. Therefore, by the definition
  of a $w$--subgroup, $\langle K,g\rangle=\KK$, so $g\in\KK$.
\end{proof}

Uniqueness in the case $\pi(w)=1$ extends to finiteness in general,
and the finite list of $w$--subgroups is computable. Here we deduce
computability from Whitehead's algorithm
\cite{whitehead_equivalent_1936}; an alternative algorithm is
described in \cite[Appendix A]{puder_primitive_2014}.

\begin{lemma}\label{lem: Finitely many w-subgroups}
  There are only finitely many $w$--subgroups in a free group
  ${\FF}$, and there is an algorithm that lists them.
\end{lemma}
\begin{proof}
  If ${\FF}$ is the fundamental group of a based finite graph
  ${\ambient}$, then any finitely generated subgroup $\KK$ can be
  realized by a based immersion of finite core graphs
  $\Lambda\immerses {\ambient}$, and if $w$ is contained in $\KK$ then
  the immersion $w\colon S\to {\ambient}$ lifts to $\Lambda$.  We only
  need to consider subgroups $\KK$ for which $w$ is not contained in a
  proper free factor, and for such subgroups $\KK$, every edge of
  $\Lambda$ is in the image of $w$. Unless $w$ is primitive in $\KK$,
  every edge of $\Lambda$ is hit at least twice by $w$, so we only
  need to consider the finitely many based immersions
  $\Lambda\immerses {\ambient}$ with $\vert \Lambda\vert\leq\vert
  w\vert/2$. For each such $\Lambda\immerses {\ambient}$, Whitehead's
  algorithm decides whether or not $w$ is primitive in $\KK$. Keep
  those $\Lambda$ of minimal rank, and of these the $w$--subgroups are
  the maximal ones with respect to inclusion: $\KK\leq\KK'$ if and
  only if the based immersion $\Lambda\to {\ambient}$ factors through
  the based immersion $\Lambda'\to {\ambient}$, which can be checked
  trivially.
\end{proof}

If we realize ${\FF}$ as the fundamental group of a core graph
${\ambient}$ and $w$ by an immersion $w\colon S\to {\ambient}$ then
each of the finitely many $w$--subgroups $\KK_i$ is realized by an
immersion of core graphs $\Lambda_i\immerses {\ambient}$.  We may then
define complexes $Q_i=\Lambda_i\cup_{w} D$ (where $w$ is the unique,
by Lemma~\ref{lem: w-subgroups of free groups are malnormal}, lift of
$w$ to $\Lambda_i$), which come equipped with immersions $Q_i\immerses
X$. These play a key role in the classification of immersions
$Y\immerses X$ with $\chi(Y)=2-\pi(w)$.

\begin{definition}\label{defn: w-subgroups of O}
  If $\KK_i\leq{\FF}$ is a $w$--subgroup we also call
  $P_i=\KK_i/\ncl{w}$ a \emph{$w$--subgroup} of $\GG={\FF}/\ncl{w}$.
\end{definition}

The $w$--subgroups come equipped with homomorphisms $P_i\to \GG$
induced by the immersions $Q_i\immerses X$.  The name `$w$--subgroup'
turns out to be justified, since by Theorem \ref{thm: Q is a subgroup}
these homomorphisms are injective.

\begin{remark}\label{rem: w-subgroups don't split freely}
  Whenever a one-relator group $\FF/\ncl{w}$ splits freely, the word
  $w$ is conjugate into a free factor of $\FF$ \cite[Proposition
    II.5.13]{lyndon-schupp}. In particular, every $w$-subgroup
  $P_i=\KK_i/\ncl{w}$ is at most one-ended, since $w$ is not contained
  in a proper free factor of $\KK_i$.  (Note that $\ZZ$ is a HNN
  extension of the trivial group, and in particular splits freely.)
\end{remark}

\subsection{Nielsen reduction}

\label{subsection: nielsen equivalence}

This section introduces the strong version of homotopy equivalence
that plays a role in our main results.

\begin{definition}\label{defn: Nielsen reduction}
 Let $X,X'$ be combinatorial 2-complexes.  A \emph{Nielsen
   equivalence} between $X$ and $X'$ is a homotopy equivalence
 $f:G_X\to G_{X'}$ and a homeomorphism $s:S_X\to S_{X'}$ such that
 $f\circ w_X\simeq w_{X'}\circ s$. (Here, we use the notation of
 Definition \ref{def: Combinatorial complex}.)  In this case, we say
 that $X$ and $X'$ are \emph{Nielsen equivalent}.
  
Let $Y$ be another 2-complex. We say that $X$ \emph{Nielsen reduces to
  $Y$} if $X$ is Nielsen equivalent to a wedge $Y\vee \bigvee_{i}
D^2_i$, where the $D^2_i$ are 2-discs with the standard cellular
structure.
\end{definition}

Complexes that Nielsen reduce to graphs can also be characterized
algebraically. The following theorem is an easy consequence of the
fact that any pair of bases of a free group are related by Nielsen
moves \cite[Proposition I.4.1]{lyndon-schupp}.
\begin{proposition}
  \label{prop: reducible iff basis}
  A two-complex $Y$ Nielsen reduces to a graph if and only if the
  conjugacy classes represented by the attaching maps for the
  two-cells of $ Y$ have representatives which are a sub-basis of the
  free group $\pio{Y^{(1)}}$.
\end{proposition}

We will make use of the following technical fact about Nielsen
reduction. It is an immediate consequence of Proposition \ref{prop:
  reducible iff basis}, because the pullback of a sub-basis along an
immersion is a sub-basis.

\begin{lemma}  \label{lem: nielsenreduces}
  Let $U,Y$ be 2-complexes. If $U$ immerses in $Y$ and $Y$ Nielsen
  reduces to a graph, then $U$ Nielsen reduces to a graph.  In
  particular, if $Y$ is one-relator and the attaching map is along a
  primitive element, then $U$ Nielsen reduces to a graph.
\end{lemma}

\subsection{One-relator pushouts and primitivity rank}

We can now classify immersions of finite complexes $Y\immerses X$ when
$\chi(Y)$ is sufficiently large: specifically, when $\chi(Y)\geq
2-\pi(w)$.

\begin{lemma}
  \label{lem: Negative immersions}
  Let $\GG={\FF}/\ncl{w}$ be a one-relator group as above, and $X$ a
  presentation complex of $\GG$, with $w$ represented by an immersion
  $w\colon S\immerses {\ambient}$. Let $Y\immerses X$ be an immersion
  from a compact connected one- or two-complex $Y$ to $X$.  Suppose
  that $\chi(Y)\geq 2-\pi(w)$, that $Y$ has no free faces, and that
  the one-skeleton of $Y$ is a core graph.
\begin{enumerate}[(i)]
\item If $\chi(Y)>2-\pi(w)$ then $Y$ reduces to a graph.
\item If $\chi(Y)=2-\pi(w)$ then either $Y$ reduces to a graph or
  $Y\immerses X$ factors through some $Q_i\immerses X$.
\end{enumerate}
\end{lemma}

\begin{proof}
  Since $Y$ has no free faces, Corollary~\ref{monotonicity} implies
  that $\chi(\hatt Y^I)\geq \chi(Y)$, where $\hatt Y^I$ is the
  immersed one-relator pushout of $Y$ (Definition~\ref{def:
    one-relator pushout}.)
  
  We first prove item (i). Suppose that $\chi(Y)>2-\pi(w)$. If
  $\pio{\folded}$ is the subgroup of ${\FF}$ corresponding to the
  1-skeleton of $\hatt{Y}^I$,
  \[
  \rk{\pio{\folded}}=2-\chi(\hatt Y^I)\leq 2-\chi(Y)<\pi(w)~.
  \]
  Since $\pio{\folded}$ is a subgroup of ${\FF}$ of rank less than
  $\pi(w)$, $w$ represents a primitive element of $\pio{\folded}$, so
  $Y$ reduces to a graph, by Lemma~\ref{lem: nielsenreduces}.
  
  The proof of item (ii) is similar. If $Y$ is a graph there is
  nothing to prove. If $w$ is primitive in $\pio{\folded}$ then, as in
  the previous paragraph, $Y$ reduces to a graph.  Otherwise,
  $\rk{\pio{\folded}}=\pi(w)$ and $w$ is not primitive in
  $\pio{\folded}$, so there is a $w$--subgroup $\KK_{i}$ of ${\FF}$
  containing $\pio{\folded}$. Since $Y^{(1)}$ is a core graph,
  ${\folded}$ is also a core graph, and so the immersion
  ${\folded}\immerses {\ambient}$ factors through
  $Q_i^{(1)}=\Lambda_i$.  Therefore $\hatt Y^I\immerses X$ factors
  through $Q_{i}\immerses X$, and so $Y\immerses X$ also factors
  through $Q_{i}$.
\end{proof}

\subsection{Homomorphisms from finitely generated groups}

In this section we combine the observations from the previous
subsections and finally prove Theorem~\ref{thm: f.g. subgroups}.  The
first lemma provides a tool for promoting results about immersions to
results about subgroups.

\begin{lemma}\label{lem: Folding 2-complexes}
  A combinatorial map of finite 2-complexes $X\to Y$ factors as 
  \[
    X\to Z\immerses Y
  \]
  where $X\to Z$ is surjective and $\pi_1$-surjective.
\end{lemma}
\begin{proof}
This is part of \cite[Lemma 4.1]{louder_one-relator_2018}.
\end{proof}

This has the following useful consequence.

\begin{lemma}\label{lem: immersing fp subgroups}
  Let $Y$ be a finite 2-complex, and let $f\colon H\to\pio{Y}$ be a
  homomorphism from a finitely presented group. Then there is a an
  immersion from a finite, connected 2-complex $g\colon Z\immerses Y$
  and a surjection $h\colon H\to \pio{Z}$ such that $f=g_*\circ h$.
\end{lemma}

\begin{proof}
  Let $\langle x_1,\ldots, x_m\mid r_1,\ldots, r_n\rangle$ be a finite
  presentation for $H$. Let $R \to Y$ be a combinatorial map from a
  rose $R$ with petals corresponding to the $x_i$.  Each relator $r_j$
  is the boundary of a singular disc diagram $D_j\to Y$.  Let $X$ be
  constructed by gluing the $D_j$ to $R$ along their boundaries.
  There is a combinatorial map $X\to Y$ realizing the homomorphism
  $f$.  Applying Lemma \ref{lem: Folding 2-complexes}, $X\to Y$
  factors through an immersion $Z\immerses Y$.
\end{proof}

For homomorphisms from finitely generated groups, we obtain the
following, slightly weaker, result.

\begin{lemma}\label{lem: Immersing fg subgroups}
  Let $Y$ be a finite 2-complex, and let $f\colon H\to\pio{Y}$ be a
  homomorphism from an $n$--generator group. There is a sequence of
  $\pi_1$-surjective immersions of finite, connected 2-complexes
  without free faces
  \[
    Z_0\immerses Z_1\immerses\cdots \immerses Z_i\immerses\cdots~,
  \]
  an immersion $g$ from the direct limit $Z=\varinjlim Z_i$ into $Y$
  and a $\pi_1$-surjection $h\colon H\to\pio{Z}$ such that $f=g_*\circ
  h$.  Furthermore, we may take $\rk{\pio{Z_0}}\leq n$.
\end{lemma}

\begin{proof}
 The existence of the sequence of immersions is \cite[Lemma
   4.4]{louder_one-relator_2018}, and the final assertion about
 $\rk{\pio{Z_0}}$ is an immediate consequence of its proof.
\end{proof}

In general, when one applies Lemma \ref{lem: Folding 2-complexes}
there may be no relation between the Euler characteristics of the
complexes $X$ and $Z$. However, we will obtain some control using a
theorem of Howie.  Recall that a group is \emph{locally indicable} if
every non-trivial finitely generated subgroup has infinite
abelianization.

\begin{thm}[{\cite[Corollary~4.2]{howie-pairs}}]\label{thm: Howie}
  If $X$ is a 2-complex and $Y\subseteq X$ is a connected subcomplex
  such that $\pio{Y}$ is locally indicable and $H_2(X,Y)=0$ then the
  map $\pio{Y}\to\pio{X}$ induced by inclusion is injective.
\end{thm}

We use Howie's theorem to prove the following lemma, which can also be
deduced from earlier results of Stallings
\cite[p171]{stallings-homology}.

\begin{lemma}\label{homologylemma}
  If $X$ is a {connected 2-complex} and $\pio{X}$ is
  generated by $n$ elements, then
  \[
  \chi(X)\geq 1-n
  \]
  with equality {only if} $\pio{X}$ is free on $n$ generators.
\end{lemma}
\begin{proof}
  Let $x_1,\ldots,x_n$ be a generating set for $\pio{X}$.  Since $X$
  is 2-dimensional and $b_1(X)\leq n$ it is clear that $\chi(X)\geq
  1-n$, so it suffices to show $\pio{X}$ is free on the $x_i$ if
  $\chi(X)=1-n$.  We can realize the $x_i$ by a combinatorial
  $\pi_1$-surjection of a rose $f\colon R\to X$. Let $M$ be the
  mapping cylinder of $f$, a 2-complex homotopy-equivalent to $X$. If
  $\chi(M)=1-n$ then $H_2(M)=0$ and the natural map $H_1(R)\to H_1(M)$
  is injective. Therefore, by the long exact sequence of a pair,
  $H_2(M,R)=0$ and so by Howie's theorem, $\pio{R}\to\pio{M}$ is
  injective, since free groups are locally indicable. Therefore,
  $\pio{M}=\pio{X}$ is free on the $x_i$.
\end{proof}

{
\begin{remark}\label{rem: One-relator rank}
Lemma \ref{homologylemma} quickly implies a classical result of Magnus
\cite{magnus_uber_1939}. In the one-relator case (which Magnus
attributes to Dehn), the result is as follows: if $G=F_n/\ncl{w}$,
then either $G\cong F_{n-1}$ or $\rk{G}=n$ \cite[Proposition
  II.5.11]{lyndon-schupp}.
\end{remark}
Lemma \ref{homologylemma} also enables us to prove the group-theoretic
analogue of Lemma~\ref{lem: Negative immersions}, from which Theorem
\ref{thm: k-free} follows immediately.}

\begin{lemma}\label{lem: Universal property of w-subgroups}
  Let $\GG={\FF}/\ncl{w}$ be a one-relator group with $\pi(w)>1$, and
  let $f\colon H\to\GG$ be a homomorphism from a finitely generated
  group $H$.
  \begin{enumerate}[(i)]
  \item If $\rk{H}<\pi(w)$ then $f$ factors through a free group.
  \item If $\rk{H}=\pi(w)$ and $H$ is not free of rank $\pi(w)$
    then either $f$ factors through a free group or $f(H)$ is
    conjugate into some $w$--subgroup $P_k$.
\end{enumerate}
\end{lemma}
\begin{proof}
By Lemma \ref{lem: Immersing fg subgroups}, there is a sequence of
$\pi_1$-surjective immersions of finite, connected 2-complexes without
free faces
\[
Z_0\immerses Z_1\immerses\cdots\immerses Z_i\immerses\cdots
\]
so that $f$ factors through $\pio{Z}$, where $Z=\varinjlim Z_i$.
Therefore, if $f$ does not factor through a free group, $\pio{Z}$ is
not free.  Since free groups are Hopfian, $\pio{Z_i}$ is not free for
all but finitely many $i$, and so we may assume without loss of
generality that $\pio{Z_i}$ is not free for any $i$.

If $\rk{H}<\pi(w)$ then, for all $i$, 
\[
\chi(Z_i)\geq 2-\rk{H}>2-\pi(w)
\]
by Lemma \ref{homologylemma}, and so $Z_i$ Nielsen reduces to a graph
by Lemma \ref{lem: Negative immersions}, which contradicts the
assumption that $\pio{Z_i}$ is not free. This proves item (i).

If $\rk{H}=\pi(w)$ then, similarly, $\chi(Z_i)\geq 2-\pi(w)$ for all
$i$, and since $\pio{Z_i}$ is not free, we must have
$\chi(Z_i)=2-\pi(w)$. Therefore, by Lemma \ref{lem: Negative
  immersions}, each immersion $Z_i\immerses X$ factors through some
$Q_{k(i)}\immerses X$. Since there are only finitely many $Q_k$ by
Lemma \ref{lem: Finitely many w-subgroups}, there is a $k$ such that
$Z_i\immerses X$ factors through $Q_k$ for infinitely many $i$, whence
$f$ factors through $P_k$. This proves item (ii).
\end{proof}

\subsection{$w$--subgroups are subgroups}
 
At last we can prove, as claimed, that the $w$-subgroups $P_i$ really
are subgroups of the one-relator group $\GG$. Recall that the maps
$P_i\to G$ are induced by immersions of one-relator complexes
$Q_i\immerses X$.

\begin{theorem}\label{thm: Q is a subgroup}
  Let ${\FF}$ be a free group with $w\in {\FF}$. The natural maps
  $P_i\to\GG$ are injective.
\end{theorem}

\begin{proof}
  We assume that $w$ is nontrivial and that $\pi(w)>1$, since the case
  $\pi(w)=1$ is well-known, as noted in Example \ref{ex: w-subgroups
    of powers}.
  
  Let $\gamma\colon S^1\to Q_i$ be an edge loop whose image in $X$ is
  null-homotopic. Let $D$ be a van Kampen diagram for $\gamma$. Let
  $R=Q_i\cup_\gamma D$, which comes equipped with a natural map $R\to
  X$. By Lemma \ref{lem: Folding 2-complexes}, this factors as
  \[
    R\to Z\immerses X
  \]
  with $R\to Z$ a $\pi_1$-surjection; in particular, we obtain a
  $\pi_1$-surjection $Q_i\immerses Z$. The complex $Z$ retracts to a
  subcomplex $Y\subseteq Z$ without free faces, and since $Q_i$ has no
  free faces the immersion $Q_i\to X$ factors through the retraction
  to $Y$. Now, $H=\pio{Y}$ is generated by $\pi(w)$ elements and is
  not free of rank $\pi(w)$ since it is a quotient of $P_i$, so by
  Lemma \ref{homologylemma}, $\chi(Y)\geq 2-\pi(w)$.  Therefore, by
  Lemma \ref{lem: Negative immersions}, either $Y$ reduces to a graph
  or it factors through some immersion $Q_j\immerses X$.  But the
  immersion $Q_i\immerses X$ factors through the immersion
  $Q_i\immerses Y$, so by Lemma~\ref{lem: nielsenreduces}, if $Y$
  reduces to a graph then $Q_i$ does too, contradicting the definition
  of a $w$--subgroup.  Therefore $Y\immerses X$ factors through some
  $Q_j$. It follows that $\KK_i\leq\KK_j$ (where these are the
  $w$--subgroups of ${\FF}$ corresponding to $Q_i$ and $Q_j$
  respectively) so, by the definition of a $w$--subgroup, $i=j$ and
  $Q_i\to Q_j$ is an isomorphism.  Therefore, $R$ retracts to $Q_i$,
  so $\gamma$ was already null-homotopic in $Q_i$.  This proves the
  theorem.
\end{proof}

Using Remark \ref{rem: w-subgroups don't split freely}, we see that
$\pi(w)$ is an invariant of the isomorphism type of the one-relator
group $\GG$.

\begin{corollary}
  If $w\in{\FF}$ is a word in a free group then $\pi(w)$ is the
  minimal rank of a non-free subgroup of the one-relator group
  $\GG={\FF}/\ncl{w}$.
\end{corollary}

\subsection{The case $\pi(w)=2$}

As explained in the introduction, the results of the previous section
show that, when $\pi(w)>2$, the subgroup structure of
$\GG={\FF}/\ncl{w}$ is like the subgroup structure of a hyperbolic
group. In this section, we examine the case $\pi(w)=2$, and notice
that the non-negatively curved behaviour of $\GG$ is concentrated in a
particular subgroup.  This follows from the next result, which shows
that in this case there is a unique $w$--subgroup of ${\FF}$.

\begin{proposition}
  \label{prop: Maximal rank-two}
  Let ${\FF}$ be a free group and $w\in {\FF}\smallsetminus 1$ an
  imprimitive element that is not a proper power. If $\HH_1$ and
  $\HH_2$ are rank-two subgroups of ${\FF}$ with $w$ contained in, but
  not primitive in, both $\HH_1$ and $\HH_2$, then $\langle
  \HH_1,\HH_2\rangle$ also has rank two.
  
  If $\pi(w)=2$ then there is a unique $w$-subgroup of $\FF$.
\end{proposition}
\begin{proof}
  Since $w$ is imprimitive and not a proper power, Theorem
  \ref{introthm: main} applies to give
  \[
    1\leq
    (\rk{\HH_1}-1)+(\rk{\HH_2}-1)-(\rk{\langle\HH_1,\HH_2\rangle}-1)~,
  \]
  and since $\rk{\HH_1}=\rk{\HH_2}=2$, it follows that
  $\rk{\langle\HH_1,\HH_2\rangle}=2$ as required.
  
  Suppose that $\pi(w)=2$. Let $\mathcal{H}=\{\HH_i\}$ be the set of
  rank-two subgroups of ${\FF}$ so that $w\in\HH_i$ and $w$ is not
  primitive in $\HH_i$; $\mathcal{H}$ is finite by Lemma \ref{lem:
    Finitely many w-subgroups}, and since $\pi(w)=2$, $\mathcal{H}$ is
  non-empty.  Considering the partial order on $\mathcal{H}$ given by
  inclusion, the previous paragraph now implies that each pair has an
  upper bound, and it follows that $\mathcal{H}$ has a unique maximal
  element $K$, which is necessarily the unique $w$--subgroup.
\end{proof}

Therefore, in this case, we drop the unnecessary subscript $i$ and
write $P$ for the $w$--subgroup of $G$.  In light of
Conjecture~\ref{conj: rel hyp} we make the following definition.

\begin{definition}\label{defn: NN associate}
  If $\pi(w)=2$ then $P$ is the \emph{peripheral subgroup} of $\GG$.
\end{definition}

We do not currently know how to prove that $P$ is uniquely defined in
$\GG$ up to isomorphism. However, Lemma \ref{lem: Universal property
  of w-subgroups} shows that if
$\GG\cong\FF/\ncl{w}\cong\FF'/\ncl{w'}$ are isomorphic then the
corresponding peripheral subgroups $P$ and $P'$ are conjugate into
each other, which somewhat justifies the term `peripheral'. If
Conjecture \ref{conj: rel hyp} held then $P$ would be malnormal in
$\GG$, and therefore would be a well-defined isomorphism invariant.

\bibliographystyle{amsalpha}

\providecommand{\bysame}{\leavevmode\hbox to3em{\hrulefill}\thinspace}
\providecommand{\MR}{\relax\ifhmode\unskip\space\fi MR }
\providecommand{\MRhref}[2]{%
  \href{http://www.ams.org/mathscinet-getitem?mr=#1}{#2}
}
\providecommand{\href}[2]{#2}

\Addresses

\end{document}